\documentclass{article}
\usepackage{amsmath,amssymb,amsthm,tikz-cd,bm,mathrsfs,hyperref}
\usepackage{dutchcal}
\usepackage[shortlabels]{enumitem}
\usepackage{graphicx}
\usepackage[margin=1in]{geometry}

\usepackage{indentfirst}

\hypersetup{
	colorlinks=true,
	linkcolor=blue,
	filecolor=magenta,
	urlcolor=cyan,
	citecolor=blue,
}

\theoremstyle{theorem}
\newtheorem*{theorem*}{Theorem}
\newtheorem*{corollary*}{Corollary}
\newtheorem*{thmdef*}{Theorem/Definition}

\theoremstyle{definition}
\newtheorem{theorem}{Theorem}[section]
\newtheorem{corollary}[theorem]{Corollary}
\newtheorem{lemma}[theorem]{Lemma}
\newtheorem{remark}[theorem]{Remark}

\newtheorem{example}{Example}
\newtheorem{definition}[theorem]{Definition}
\newtheorem{proposition}[theorem]{Proposition}
\newtheorem{Setup}{Setup}
\newtheorem{thmletter}{Theorem}

\newcommand{\mc}[1]{\mathcal{#1}}
\newcommand{\mr}[1]{\mathrm{#1}}
\newcommand{\mf}[1]{\mathfrak{#1}}
\newcommand{\mb}[1]{\mathbb{#1}}

\newcommand{\Z}[1]{\texttt{#1}}

\DeclareMathOperator{\knr}{Ker\,}
\DeclareMathOperator{\im}{Im\,}
\newcommand{\ann}[1]{\mathrm{ann}_{#1}}

\newcommand{\cc}{\mathbb{C}}
\newcommand{\nt}{\mathbb{N}}
\newcommand{\nz}{\mathbb{N}^+}
\newcommand{\kk}{\mathbb{K}}
\newcommand{\z}{\mathbb{Z}}
\newcommand{\Q}{\mathbb{Q}}
\newcommand{\zp}{\mathbb{Z}_p}
\newcommand{\fp}{\mathbb{F}_p}
\newcommand{\bv}[1]{\mathbb{V}(#1)}
\newcommand{\nm}[1]{\Vert #1\Vert}
\newcommand{\wno}[1]{|#1|_{\B{w}}}

\newcommand{\ideala}[1]{\langle{#1}\rangle}
\newcommand{\tre}[1]{\langle{#1}\rangle_e}
\newcommand{\dit}[1]{^{(#1)}}

\newcommand{\dpot}[1]{\partial^{[#1]}_{\mathbf{t}}}
\newcommand{\dpox}[1]{\partial^{[#1]}_{\mathbf{x}}}

\newcommand{\End}[1]{\mathrm{End}_{#1}}

\newcommand{\B}[1]{\mathbf{#1}}
\newcommand{\Y}[1]{\bm{#1}}

\newcommand{\nJI}[1]{\nu^{\mf{a}}_I(p^{#1})}
\newcommand{\ninv}[1]{\nu^\bullet_{#1}}
\newcommand{\pow}[1]{p^{#1}}
\newcommand{\ppm}[1]{^{[p^{#1}]}}
\newcommand{\ppdm}[1]{^{[\frac{1}{p^{#1}}]}}
\newcommand{\cpt}[1]{\textbf{c}^{#1}}
\newcommand{\wmd}[1]{\mr{deg}_{\B{w}}(#1)}
\newcommand{\wxd}[1]{\mr{Deg}_{\B{w}}(#1)}
\newcommand{\wtd}[1]{\mr{wt}\,(#1)}
\newcommand{\Tupp}[1]{\mr{Supp}\,(#1)}
\newcommand{\jac}[1]{\mr{Jac}\,(#1)}

\newcommand{\eul}[1]{\vartheta_{#1}}

\newcommand{\dlr}{D^{(e)}_R}
\newcommand{\dle}[1]{\textbf{D}^e_{#1}}

\DeclareMathOperator{\spec}{Spec\,}
\DeclareMathOperator{\spmc}{Spec_{max}\,}
\newcommand{\bcap}[1]{\bigcap_{#1}}
\newcommand{\bcup}[1]{\bigcup_{#1}}
\newcommand{\bops}[1]{\bigoplus_{#1}}
\newcommand{\ceil}[1]{\lceil #1 \rceil}
\newcommand{\ott}[1]{\otimes_{#1}}

\newcommand{\ssm}[1]{\sum_{#1}}
\newcommand{\tld}[1]{\widetilde{#1}}
\newcommand{\Lxf}{\bm{f_1^{s_1}\cdots f_r^{s_r}}}
\newcommand{\bxf}{\bm{f^s}}

\newcommand{\ud}[1]{_{#1}}
\newcommand{\kud}[1]{_{(#1)}}

\newcommand{\cut}[1]{\textbf{C}(\zp,#1)}
\newcommand{\rcut}[1]{\textbf{C}(\zp^r,#1)}
\newcommand{\ceut}[1]{\textbf{C}^e(\zp,#1)}
\newcommand{\rceut}[1]{\textbf{C}^e(\zp^r,#1)}

\newcommand{\mustata}{Musta\c{t}\u{a}\,}
\newcommand{\QUG}{Quinlan-Gallego\,}

\newcommand{\Nft}{$F$-thresholds\,}
\newcommand{\Nbs}{Bernstein-Sato\,}
\newcommand{\Ndpo}{differential operator\,}
\newcommand{\Ndsp}{differential operators\,}
\newcommand{\Nwh}{weighted homogeneous\,}

\usepackage[hang,flushmargin]{footmisc}
\usepackage{lettrine}

\setlist[enumerate]{itemsep=.5pt, topsep=1pt}
\setlist[itemize]{itemsep=.5pt, topsep=1pt}

\begin{document}
	\title{Bernstein-Sato functional equations for ideals in positive characteristic}
	\author{Siyong Tao, Zida Xiao, and Huaiqing Zuo}
	\date{}
	\maketitle
	\vspace{-2.5em}
	
	\begin{abstract}
		For an ideal of a regular $\cc$-algebra, its Bernstein-Sato polynomial is the monic polynomial of the lowest degree satisfying an Bernstein-Sato functional equation. We generalize the notion of Bernstein-Sato functional equations to the case of ideals in an $F$-finite ring of positive characteristic $p$, and show the relationship between these equations and Bernstein-Sato roots.
		By applying this theory, we provide an explicit description of Bernstein-Sato roots of a weighted homogeneous polynomial with an isolated singularity at the origin in characteristic $p$. Moreover, we give multiplicative and additive Thom-Sebastiani properties for the set of Bernstein-Sato roots, which prove the characteristic $p$ analogue of Budur and Popa's question.
		
		
	\end{abstract}
	
	\section{Introduction}\label{sec:intro}
	Let $f\in A=\cc[x_1,\dots,x_n]$ be a non-constant polynomial in $n$ variables. We define the free rank one $A_f[s]$-\\module $A_f[s]\bxf$, where $\bxf$ is the formal symbol for the generator. Let $D_{A_f/\cc}$ be the ring of $\cc$-linear differen-\\tial operators on $A_f$ (see Definition \ref{def:diff oper}). $D_{A_f/\cc}[s]$-module structure on $A_f[s]\bxf$ is given by the rule
	\begin{equation*}
		\partial\cdot a(s)\bxf=\partial(a(s))\bxf+\frac{sa(s)\partial(f)}{f}\bxf
	\end{equation*}
	for any $\cc$-linear derivation $\partial$ on $A_f$ and $a(s)\in A_f[s]$.
	\Nbs polynomial of $f$, we denote it by $b_f(s)$, is the minimal monic generator of the ideal $\{b(s)\in\cc[s]:b(s)\bxf\in D_A[s]\cdot f\bxf\}\subseteq \cc[s]$.
	Here $D_{A/\cc}[s]$ can be viewed as a $\cc$-subalgebra of $D_{A_f/\cc}[s]$ (see Lemma \ref{lem:local of diff}). For any $P(s)\in D_{A/\cc}[s]$, we call $b(s)\bxf=P(s)\cdot f\bxf$ a \Nbs functional equation for $f$. Clearly, we have
	\begin{equation*}
		\ideala{b_f(s)}=\ann{\cc[s]}\,\frac{D_A[s]\cdot \bxf}{D_A[s]\cdot f\bxf}.
	\end{equation*}
	
	\Nbs polynomial was introduced by Sato and Shintani \cite{SS72}. And its existence was proved by Bernstein \cite{Ber72}. Kashiwara \cite{Kas76} showed that roots of $b_f(s)$ are negative and rational. These roots are important invariants of the hypersurface singularity in $\cc^n$ defined by $f$. They are known to be closely related to the log-canonical threshold of $(\cc^n,V(f))$ \cite{Kol97}, the jumping coefficients of $(\cc^n,V(f))$ \cite{ELSV04}, and the eigenvalues of the monodromy action on the cohomology of Milnor fiber of $f$ \cite{Mal83}.

	Budur, \mustata and Saito \cite{Bur06} generalized the notion of \Nbs polynomials to the case of a complex subvariety $Z$ embedded in a smooth affine variety $X=\spec A$. Let $I=\ideala{f_1,\dots,f_r}\subseteq A=\mc{O}_X(X)$ be the vanishing ideal of $Z$. Denote $A_{f_1\cdots f_r}$ by $A_{\B{f}}$. Given indeterminates $\B{s}=s_1,\dots,s_r$, we still define the free rank one $A_{\B{f}}[\B{s}]$-module $A_{\B{f}}[\B{s}]\Lxf$ with the formal symbol generator $\Lxf$. $D_{A_{\B{f}}/\cc}[\B{s}]$-module structure on $A_{\B{f}}[\B{s}]\Lxf$ is given by the rule
	\begin{equation*}
		\partial\cdot a(\B{s})\Lxf=\partial(a(\B{s}))\Lxf+\sum^r_{i=1}\frac{s_ia(\B{s})\partial(f_i)}{f_i}\Lxf,
	\end{equation*}
	for all $\cc$-linear derivation $\partial$ on $A_{\B{f}}$ and $a(\B{s})\in A_{\B{f}}[\B{s}]$.
	\Nbs polynomial of $I$, denoted by $b_I(s)\in\cc[s]$, is defined to be the monic polynomial of smallest degree satisfying the \Nbs functional equation
	\begin{equation}\label{eq:BS-equ in cpx}
		b(s_1+\cdots+s_r)\Lxf=\ssm{\B{c}=(c_1,\dots,c_r)\in\z^r,\sum^r_{i=1}c_i=1}P_{\B{c}}(\B{s})\cdot\prod\ud{c_i<0}\binom{s_i}{-c_i}\prod^r_{i=1}f_i^{c_i}\Lxf,
	\end{equation}
	where $P_{\B{c}}(\B{s})\in D_{A/\cc}[\B{s}]$. $b_I(s)$ is independent of the choice of generators of $I$. 
	\begin{theorem}[\text{[\textcolor{blue}{BMS06}, Theorem 2]}]\label{thm:bs roots and jumping in czero}
		The smallest root $\alpha\ud{I}$ of $b_I(-s)$ concides with the log-canonical threshold of $(X,Z)$, and jumping coefficients of $(X,Z)$ in $[\alpha\ud{I},\alpha\ud{I}+1)$ are roots of $b_I(-s)$.
	\end{theorem}
	
	Theory of $V$-filtration developed by Kashiwara \cite{Kas83} and Malgrange \cite{Mal83} is closely related to Bern-\\stein-Sato theory. It can be applied to show the existence of quite general \Nbs polynomials. 
	
	We keep the notation introduced above. Let $i_\mc{F}:X\to X\times\ud{\cc}\mb{A}^r_{\cc}$ be the graph embedding of $\mc{F}=\{f_i\}_{i=1}^r$. $i_{\mc{F}}$ is given by the homomorphism $A[\B{t}]=A[t_1,\dots,t_r]\to A, t_i\mapsto f_i,1\leq i\leq r$. Direct image $(i_{\mc{F}})_+A$ of the $D_{A/\cc}$-module $A$, as a $D_{A[\B{t}]/\cc}$-module, is the local cohomology module
	\begin{equation}\label{eq:local coho}
		H_I=H^r_{\ideala{f_i-t_i:1\leq i\leq r}}A[\B{t}]=\frac{A[\B{t}]_{\prod^r_{i=1}(f_i-t_i)}}{\im(\bigoplus^r_{i=1}A[\B{t}]_{\prod\ud{j\neq i}(f_j-t_j)}\to A[\B{t}]_{\prod^r_{i=1}(f_i-t_i)})}.
	\end{equation}
	The $V$-filtration along the ideal $\ideala{\B{t}}=\ideala{t_1,\dots,t_r}$ of $A[\B{t}]$ on $D_{A[\B{t}]/\cc}$ is the filtration indexed by integers $i\in\z$ defined by $V^iD_{A[\B{t}]/\cc}=\{\varphi\in D_{A[\B{t}]/\cc}:\varphi\cdot \ideala{\B{t}}^j\subseteq \ideala{\B{t}}^{j+i}\;\text{for all}\;j\in\z\}$, where we let $\ideala{\B{t}}^j=A[\B{t}]$ for all $j\leq 0$. $b_I(s)$ is the minimal polynomial of the action of $\vartheta=-\sum^r_{i=1}\partial\ud{t_i}t_i$ on
	\begin{equation}\label{eq:NF in comp}
		N_I=\frac{V^0D_{A[\B{t}]/\cc}\cdot\delta}{V^1D_{A[\B{t}]/\cc}\cdot\delta},\quad\text{where}\; \delta=\left[\frac{1}{\prod^r_{i=1}(f_i-t_i)}\right]\in H_I.
	\end{equation}
	We refer readers to \cite{Bur06} and \cite{Mon21} for more details. Assume that the set of roots of $b_I(s)$ without counting multiplicities is $\{\lambda\ud{1},\dots,\lambda\ud{l}\}$, then $N_I$ decomposes as a finite direct sum $N_I=\bigoplus^n_{i=1}(N_I)_{\lambda\ud{i}}$, where $(N_I)_{\lambda\ud{i}}$ is the $\lambda\ud{i}$-generalized eigenspace of $\vartheta$.
	
	We now link \Nbs functional equations and the minimal polynomial of the action of $-\vartheta$ on $N_I$. The general philosophy is as follows. Suppose that all $f_i$ are non-zero divisors of $R$. $H'_I=H^r_{\ideala{f_i-t_i:1\leq i\leq r}}A_{\B{f}}[\B{t}]$ admits a $D_{A_{\B{f}}/\cc}[\B{s}]$-module structure. $H_I$ can be seen as a $D_{A/\cc}[\B{s}]$-submodule of $H'_I$. There exists an $D_{A_{\B{f}}/\cc}[s]$-module isomorphism
	\begin{equation*}
		\Theta\ud{I}:A_{\B{f}}[\B{s}]\Lxf\mapsto H'_I,\quad s_i\Lxf\mapsto -\partial\ud{t_i}t_i(\delta).
	\end{equation*}
	The inverse image of $V^1D_{A[\B{t}]/\cc}\cdot \delta\subseteq H'_I$ under $\Theta\ud{I}$ is exactly the $D_{A/\cc}[\B{s}]$-module
	\begin{equation*}
		\ssm{(c_1,\dots,c_r)\in \z^r:\sum^r_{j=1}c_j=1}D_{A/\cc}[\B{s}]\cdot\prod\ud{c_i<0}\binom{s_i}{-c_i}\prod^r_{i=1}f_i^{c_i}\Lxf.
	\end{equation*}
	This yields that $b(s)\in\cc[s]$ satisfies a \Nbs functional equation \eqref{eq:BS-equ in cpx} for $I$ if and only if the action of $b(\vartheta)$ on $N_I$ is zero. 
	
	Notation of \Nbs functional equation in characteristic zero is not suitable for consideration about an element $f$ in a regular $\kk$-algebra $R$ with $\kk$ a perfect field of characteristic $p>0$, since $D_{R_f/\kk}$ can not be generated by derivations on $R_f$. It is hard to define a $D_{R_f/\kk}[s]$-module structure on $R_f[s]\bxf$. Moreover, generalized test ideals, (see Definition \ref{def:test ideals}), are related to multiplier ideals (see \cite{Laz04}) via reduction modulo $p$. Naturally, $F$-jumping exponents (see Definition \ref{def:Fjp}) are characteristic-$p$ analogues of jumping coefficients.
	Bernstein-Sato polynomial over $\cc$ carries information about jumping coefficients.
	\Nbs theory in positive characteristic has been developed by \mustata \cite{Mus09}, Bitoun \cite{Bit18} and \QUG \cite{Qge21}. They established the connection between ``\Nbs roots" and $F$-jumping exponents. We begin by introducing some basic notation.
	
	\begin{Setup}\label{Setup:into}
		Let $R$ be an $F$-finite regular ring of characteristic $p>0$ (See Section \ref{subsec:diff}), $I$ a non-zero ideal of $R$ with a system of generators $\mc{F}=\{f_i\}_{i=1}^r$. $R[\B{t}]=R[t_1,\dots,t_r]$. Let $D_{R[\B{t}]}$ be the ring of $\fp$-linear differential operators on $R[\B{t}]$ with the $V$-filtration $\{V^iD_{R[\B{t}]}=\{\varphi\in D_{R[\B{t}]}:\varphi\cdot\ideala{\B{t}}^j\subseteq \ideala{\B{t}}^{j+i}\;\text{for all}\;j\in\z\}\}_{i\in\z}$ along the ideal $\ideala{\B{t}}=\ideala{t_1,\dots,t_r}$. $H_{\mc{F}}$ is the local cohomology module $H^r_{\ideala{f_i-t_i:1\leq i\leq r}}R[\B{t}]$. Consider the $R[\B{t}]$-module
		\begin{equation*}
			N_{\mc{F}}=\frac{V^0D_{R[\B{t}]}\cdot\delta}{V^1D_{R[\B{t}]}\cdot\delta},\quad \delta=\left[\frac{1}{\prod^r_{i=1}(f_i-t_i)}\right]\in H_{\mc{F}}.
		\end{equation*}
	\end{Setup}
	
	\mustata observed that higher Euler operators $\{\eul{\B{t},p^i}\}_{i\in\nt}$ (see Section \ref{subsec:Higher Euler}) act on $N_{\mc{F}}$. And $N_\mc{F}$ splits as a direct sum of multi-eigenspaces.
	\QUG generalized the work of \mustata and Bitoun, and defined \Nbs roots with values in $\zp$ for $I$. A $p$-adic integer $\mc{x}=\ssm{i\in\nt}\mc{x}^{[i]}p^i\in\zp$ is a Bernstein-Sato root of $I$ if and only if the $\mc{x}$-multi-eigenspace 
	\begin{equation*}
		(N_{\mc{F}})\kud{\mc{x}}=\{u\in N_{\mc{F}}:\eul{\B{t},p^i}\cdot u=\mc{x}^{[i]}u\;\text{for all}\;i\in\nt\}\neq 0.
	\end{equation*}
	\QUG gave a characteristic $p$ analogue of Kashiwara’s rationality theorem, which states that every \Nbs root of $I$ is negative and rational (see Theorem \ref{thm:ration}).
	
	Jeffries, N\'{u}\~{n}ez-Betancourt and \QUG \cite{JNQ23} extended \Nbs theory in positive characteristic to singular rings. Let $T$ be an $\fp$-algebra. They introduced the $T$-algebra of continuous maps from the ring of $p$-adic integers $\zp$ to $T$. We denote it by $\cut{T}$. For any principal ideal $\ideala{f}$ of $R$ with $f$ a non-zero divisor. Jeffries and co-authors constructed the free rank one $\cut{R_f}$-module $\cut{R_f}\bxf$. It corresponds to the module $A_f[s]\bxf$ in the complex case. Both $\cut{R_f}\bxf$ and $H'_f=R_f[t]_{f-t}/R_f[t]$ admit  $\cut{D_{R_f}}$-module structures. And there is a $\cut{D_{R_f}}$-module isomorphism $\Theta\ud{f}:\cut{R_f}\bxf\to H'_f$. 
	
	Jeffries and co-authors proved that $\mc{x}\in\zp$ is a \Nbs root of $f$ if and only if for every $\Z{F}\in\cut{\fp}$ satisfying $\Z{F}\bxf=\Z{P}\cdot f\bxf$ for some $\Z{P}\in\cut{D_R}$ (We call this the \Nbs functional equation for $f$ in characteristic $p$), there is $\Z{F}(\mc{x})=0$. Motivated by \Nbs functional equations for ideals in complex case, in Section \ref{sec:cont zp}, we first discuss basic properties of $\cut{T}$ for a fixed $\fp$-algebra $T$. We generalize this concept to the case of the $T$-algebra of continuous maps from $\zp^r$ to $T$ for a fixed $r\in\nt$, denoted by $\rcut{T}$. Section \ref{sec:cont zp} is a building block for the satisfactory theory of \Nbs functional equations for ideals in positive characteristic. Keep notation in Setup \ref{Setup:into}. Let $R_{\B{f}}=R_{f_1\cdots f_r}$. In Section \ref{sec: BS theory}, we summarize the work of Jeffries, N\'{u}\~{n}ez-Betancourt and \QUG, and give another discription of the $(D_{R[\B{t}]})_0$-module $N_\mc{F}$. And then we obtain the \Nbs functional equations for ideals in positive characteristic. From these equations we see the relationship between \Nbs roots and \Nbs polynomials.
	
	\begin{thmletter}[Theorem \ref{thm:f1f2--}, Corollary \ref{cor:mthmA}]
		Let	$\rcut{R_{\B{f}}}\Lxf$ be the free rank-one $\rcut{R_{\B{f}}}$-module generated by the formal symbol $\Lxf$, and let $H'_\mc{F}=H^r_{\ideala{f_i-t_i:1\leq i\leq r}}R_{\B{f}}[\B{t}]$. $\rcut{R_{\B{f}}}\Lxf$ and $H'_\mc{F}$ admit $\rcut{D_{R_{\B{f}}}}$-module structures. There is a $\rcut{D_{R_{\B{f}}}}$-module isomorphism
		\begin{equation*}
			\Theta\ud{\mc{F}}:\rcut{R_{\B{f}}}\Lxf\cong H'_{\mc{F}}.
		\end{equation*}
		Suppose that all $f_i$ are non-zero divisors of $R$. $H_\mc{F}$ can be regarded as a $\rcut{D_R}$-submodule of $H'_\mc{F}$. With the identification above, $N_{\mc{F}}$ is isomorphic to the $\rcut{D_R}$-module
		\begin{equation*}
			\frac{\ssm{\alpha\in\textbf{Fun}_r,\nm{\alpha}=0}\rcut{D_R}\Z{Bd}^{[\alpha]}\B{f}^{[\alpha]}\Lxf}{\ssm{\alpha\in\textbf{Fun}_r,\nm{\alpha}=1}\rcut{D_R}\Z{Bd}^{[\alpha]}\B{f}^{[\alpha]}\Lxf}.
		\end{equation*}
		We will give precise definitions of $\textbf{Fun}_r$, $\Z{Bd}^{[\alpha]}$ and $\B{f}^{[\alpha]}$ in Section \ref{subsec:Direct}. 
		
		For a continuous function $\Z{F}\in\cut{\fp}$, we say that $\Z{F}$ satisfies a \Nbs functional equation for $I$ if the following identity holds.
		\begin{equation*}
			\Sigma\ud{r,R}(\Z{F})\Lxf=\ssm{\alpha\in\textbf{Fun}_r,\nm{\alpha}=1}\Z{P}_{\alpha}\cdot\Z{Bd}^{[\alpha]}\B{f}^{[\alpha]}\Lxf.
		\end{equation*}
		Here every $\Z{P}_\alpha\in\rcut{D_R}$, and $\Sigma\ud{r,R}(\Z{F})\in\rcut{\fp}$ sends every $(\mc{x}_1,\dots,\mc{x}_r)\in\zp^r$ to $\Z{F}(\mc{x}_1+\cdots+\mc{x}_r)$. $\mc{x}\in\\\zp$ is a \Nbs root of $I$ if and only if for any $\Z{F}\in\cut{\fp}$ satisfying a \Nbs functional equation for $I$, there is $\Z{F}(\mc{x})=0$.
	\end{thmletter}
	
	Let $\kk$ be a perfect field of characteristic $p$. To formulate \Nbs roots of any $f\in\kk[x_1,\dots,x_n]$, by Theorem \ref{thm:FJ and BSR}, we need to compute $F$-jumping exponents of $f$. Blickle, \mustata and Saito \cite{BMS08} showed that the set of $F$-jumping exponents of $f$ coincides with the set of $F$-thresholds of $f$ (see Definition \ref{def:nu-invar} and Lemma \ref{lem:FJ and Fthe}). $F$-threshold of $f$ with respect to an ideal $\mf{a}$ of $\kk[x_1,\dots,x_n]$ such that $f\in\sqrt{\mf{a}}$ is an important numerical invariant of $f$, denoted by $\cpt{\mf{a}}(f)$. Effective methods for computing general $F$-thresholds are still lacking. Consequently, computation of \Nbs roots of polynomials remains an open problem.
	
	In Section \ref{sec:WH}, we first compute $F$-threshold of a weighted homogeneous polynomial $f$ with respect to the ideal $\ideala{x_1^{a_1},\dots,x_n^{a_n}}$ for a fixed $n$-tuple $\B{a}=(a_1,\dots,a_n)\in(\nz)^n$. Under certain hypotheses, these special $F$-\\thresholds provide precise information about \Nbs roots of $f$. But these hypotheses is too strict, so for general $f$, we try to find enough $\Z{F}\in\cut{\fp}$ satisfying the \Nbs functional equations for $f$, the expression for common zeros of these $\Z{F}$ can then be determined.
	
	\begin{thmletter}[Theorem \ref{thm:BSR of spec f}, Corollary \ref{cor:thmB}, Theorem \ref{thm:BSfun of WH}]
		Let $f\in\kk[x_1,\dots,x_n]$ be a weighted homogeneous polynomial of degree $d$ with respect to the weight vector $\B{w}=(w_1,\dots,w_n)\in(\nz)^n$. Suppose that
		\begin{equation*}
			\jac{f}=\ideala{\frac{\partial f}{\partial x_1},\dots,\frac{\partial f}{\partial x_n}}\;\text{is}\;\ideala{x_1,\dots,x_n}\text{-primary}.
		\end{equation*}
		If $\lambda$ is a \Nbs root of $f$, then either $\lambda=-\frac{a}{b}\in(0,1)$ with $a,b\in\nz,p\nmid b$ and $p\mid b-a$, or $d\lambda\in\z$, and there exists a monomial $x_1^{u_1}\cdots x_n^{u_n}\notin \jac{f}$ such that $d\lambda\equiv \sum^r_{i=1}w_i(u_i+1)\bmod p$.  
		
		Furthermore, if for any integer $e\in\nt$ and $0\leq m\leq p^e-1$, the ideal $\ideala{f^m}\ppdm{e}$ (see Section \ref{subsec:Fthres}) is a monomial ideal, then the set of \Nbs roots of $f$ is given by
		\begin{equation*}
			\{-1\}\cup\{-\frac{\sum^r_{i=1}w_ia_i}{d}:(a_1,\dots,a_r)\in(\nz)^n,\cpt{\ideala{x_1^{a_1},\dots,x_n^{a_n}}}(f)=\frac{\sum^r_{i=1}w_ia_i}{d}\in(0,1)\cap\z_{(p)}\}.
		\end{equation*}
	\end{thmletter}
	
	In Section \ref{sec:TS pro}, we prove the following additive and multiplicative Thom-Sebastiani properties for Bernstein-Sato roots. The multiplicative property is a  characteristic $p$ analogue of the long-standing question asked by Budur \cite{Bud12} and Popa \cite{Popa}. For polynomials $f\in\cc[x_1,\dots,x_n]$ and $g\in\cc[y_1,\dots,y_m]$, consider $f\cdot g$ as a polynomial in $\cc[x_1,\dots,x_n,y_1,\dots,y_m]$. Do their \Nbs polynomials satisfy $b_{f\cdot g}(s)=b_f(s)\cdot b_g(s)$?
	This question has been solved by Shi-Zuo \cite{SZ24} and Lee \cite{Lee25} independently. And the additive property is a characteristic $p$ analogue of \cite[Theorem 6]{Bur06}.
	
	\begin{thmletter}[Theorem \ref{mthm:C}]
		Let $R$ and $S$ be regular $F$-finite rings of characteristic $p$. Let $I$ be an ideal of $R$ and $J$ an ideal of $S$. Set $Q=R\ott{\fp}S$. We define $\mf{A}$ (resp. $\mf{M}$) to be the ideal of $Q$ generated by elements $\{f\otimes g\}_{f\in I,g\in J}$ (resp. $\{f\otimes 1,1\otimes g\}_{f\in I,g\in J}$).
		\begin{enumerate}
			\item $\mc{z}\in\zp$ is a \Nbs root of $\mf{A}$ if and only if $\mc{z}=\mc{x}+\mc{y}$, where $\mc{x}$ (resp. $\mc{y}$) is a \Nbs root of $I$ (resp. $J$).
			\item If $I$ (resp. $J$) is not contained in the nilradical of $R$ (resp. $S$), then the set of \Nbs roots of $\mf{M}$ is the union of the set of \Nbs roots of $I$ and the set of the \Nbs roots of $J$.
		\end{enumerate}
	\end{thmletter}

	\section{Preliminaries}
	
	\subsection{Differential operators}\label{subsec:diff}
	Let $\kk$ be a field, and let $R$ be a commutative $\kk$-algebra. $\End{\kk}(R)$ denotes the ring of $\kk$-linear endomorphisms of $R$.
	
	\begin{definition}\label{def:diff oper}
		Fix an integer $n\in\nt$. $D^n_{R/\kk}$ is the $R$-module of all $\kk$-linear \Ndsp on $R$ of order $\leq n$. It can be explicitly described as follows:
		\begin{enumerate}
			\item $D^0_{R/\kk}=\End{R}(R)$. We identity $R$ with $D^0(R/\kk)$.
			\item $D^n_{R/\kk}=\{\phi\in\End{\kk}(R):[\phi,r]\in D^{n-1}_{R/\kk}\}$. Here $[\phi,r]$ is the commutator $\phi\circ r-r\circ\phi$.
		\end{enumerate}
		The ring of $\kk$-linear \Ndsp on $R$ is defined by $D_{R/\kk}=\cup\ud{n\in\nt}D^n_{R/\kk}$.
	\end{definition}
	
	Now we suppose that $\kk$ is a perfect field of prime characteristic $p$. For every $e\in\nt$, let $R^{p^e}=\{r^{p^e}:r\in R\}$ denote the subring of $p^e$-th powers of $R$. We assume that $R$ is $F$-finite, i.e., $R$ is $R^{p^e}$-finite for every $e$.
	
	\begin{lemma}[\cite{Yek92}]
		The ring of $\kk$-linear \Ndsp on $R$ is independent of the choice of the perfect ground field $\kk$. It is given by
		\begin{equation*}
			D_{R/\kk}=\bcup{e\in\nt}\dle{R},\quad \text{where}\;\dle{R}=\End{R^{p^e}}(R).
		\end{equation*}
		We call the $R$-algebra $\dle{R}$ the ring of differential operators of level $e$ on $R$. Throughout the paper, we will not emphasize the perfect field $\kk$, as every $F$-finite ring of characteristic $p$ can be seen as an $\fp$-algebra.
	\end{lemma}
	
	Let $S$ be a multiplicatively closed subset of $R$. Any element of $S^{-1}R$ has the form $\frac{r}{s^p}$, where $r\in R$ and $s\in S$. So $S^{-1}R$ is $F$-finite.
	
	\begin{lemma}[\text{[\textcolor{blue}{BJNnB19}, Remark 2.18], [\textcolor{blue}{Qui21}, Proposition 2.30]}]\label{lem:local of diff}
		There are compatible isomorphisms of $S^{-1}R$-algebras
		\begin{equation*}
			S^{-1}R\ott{R}\dle{R}\cong \dle{S^{-1}R},\quad e\in\nt.
		\end{equation*}
		They give rise to an $S^{-1}R$-algebra isomorphism $S^{-1}R\ott{R}D_R\to D_{S^{-1}R}$. Given any $\phi\in D_R$, we still denote its image under the localization map $D_R\to S^{-1}R\ott{R}D_R$ by $\phi$. Action of $\phi$ on $S^{-1}R$ can be interpreted as:
		\begin{enumerate}
			\item If $\phi\in D^0_{R/\kk}$, then $\phi(\frac{r}{s})=\frac{\phi(r)}{s}$.
			\item If $\phi\in D^n_{R/\kk}$, we assume that the action on $S^{-1}R$ has been defined for every \Ndpo on $R$ of order $\leq n-1$, then
			\begin{equation*}
				\phi\left(\frac{r}{s}\right)=\frac{\phi(r)-[\phi,s](\frac{r}{s})}{s}.
			\end{equation*}
			\item If $\phi\in\dle{R}$, then $\phi(\frac{r}{s})=\frac{\phi(s^{p^e-1}r)}{s^{p^e}}$.
		\end{enumerate}
	\end{lemma}
	
	\subsection{Higher Euler operators on the polynomial ring}\label{subsec:Higher Euler}
	Let $R$ be an $F$-finite ring of characteristic $p$. Fix an integer $r\in\nt$, consider the $F$-finite polynomial ring $R[\B{t}]=R[t_1,\dots,t_r]$. We introduce the following notation.
	
	Let $\B{a}=(a_1,\dots,a_r),\B{b}=(b_1,\dots,b_r)$ be $r$-tuples of $\nt$.
	\begin{itemize}
		\item Define the partial order on $r$-tuples: $\B{a}\leq\B{b}$ if and only if  if $a_i\leq b_i$ for all $1\leq i\leq r$. Let $\B{0}_r=(0,\dots,0)$ and $\B{1}_r=(1,\dots,1)$. The set $\textbf{N}_r(e)$ consists of all $r$-tuples $\B{a}\in\nt^r$ such that $\B{a}\leq (p^e-1)\B{1}_r$.
		\item The generalized binomial coefficient $\binom{\B{b}}{\B{a}}$ is $\prod^r_{i=1}\binom{b_i}{a_i}$.
		\item Let $u\in R$, $u\B{t}^{\B{a}}$ denotes the monomial $ut_1^{a_1}\dots t_r^{a_r}$. If $u\neq 0$, degree of $u\B{t}^{\B{a}}$ is the integer $|\B{a}|=\sum^r_{i=1}a_i$.
	\end{itemize}
	
	For each $\phi\in D_R$ and $\B{a}\in\nt^r$, the divided power differential operator $\phi\,\dpot{\B{a}}$ acts on $R[\B{t}]$ by
	\begin{equation*}
		\phi\,\dpot{\B{a}}:u\B{t}^{\B{b}}\mapsto
		\begin{cases}
			\phi(u)\binom{\B{b}}{\B{a}}\B{t}^{\B{b}-\B{a}},&\text{if}\;\B{a}\leq\B{b},\\
			0,&\text{otherwise}.
		\end{cases}
	\end{equation*}
	An operator $\phi$ on $R$ can be seen as the operator $\phi\,\dpot{\B{0}_r}$ on $R[\B{t}]$. If $\phi\in\dle{R}$ and $\B{a}\in\textbf{N}_r(e)$, then $\phi\,\dpot{\B{a}}\in\dle{R[\B{t}]}$ (This comes from Lucas's Theorem \eqref{eq:luc}). Usually we take $\phi=1$ (the identity map). 
	
	\begin{lemma}[\text{[\textcolor{blue}{JNnBQG23}, Lemma 2.13]}]\label{lem:V-fil in pc}
		For each $e\in\nt$, $\dle{R[\B{t}]}$ is a $\dle{R}$-module generated by 
		\begin{equation*}
			\{\dpot{\B{a}}\B{t}^{\B{b}}\}_{\B{a}\in \textbf{N}_r(e),\B{b}\in\nt^r}.
		\end{equation*}	
		Every $\varphi\in\dle{R[\B{t}]}$ can be written as an $\fp$-linear combination of operators of the form $\phi\,\dpot{\B{a}}\B{t}^{\B{b}}$ with $\phi\in\dle{R}$, $\B{a}\in \textbf{N}_r(e)$ and $\B{b}\in \nt^r$.
		
		Fix an integer $i\in\z$, let $(D_{R[\B{t}]})_i$ be the $D_R$-submodule of $D_{R[\B{t}]}$ generated by $\{\dpot{\B{a}}\B{t}^{\B{b}}\}_{\B{a},\B{b}\in\nt^r,|\B{b}|-|\B{a}|=i}$. The family $\{(D_{R[\B{t}]})_i\}_{i\in\z}$  characterizes the $V$-filtration along $\ideala{\B{t}}=\ideala{t_1,\dots,t_r}$ as
		\begin{equation*}\label{eq:Vfil and deg}
			V^iD_{R[\B{t}]}=\{\varphi\in D_{R[\B{t}]}:\varphi\cdot\ideala{\B{t}}^j\subseteq \ideala{\B{t}}^{j+i}\;\text{for all}\;j\in\z\}=\bops{j\geq i}(D_{R[\B{t}]})_j.
		\end{equation*}
		When $i\geq 0$, it follows that $V^iD_{R[\B{t}]}=(D_{R[\B{t}]})_0\cdot\ideala{\B{t}}^i$. 
	\end{lemma}
	
	$(D_{R[\B{t}]})_0$ plays an significant role. For each $k\in\nt$, the $k$-th Euler operator $\eul{\B{t},k}$ is the $R$-linear operator
	\begin{equation*}
		\ssm{\B{a}\in\nt^r:|\B{a}|=k}(-1)^k\dpot{\B{a}}\B{t}^{\B{a}}.
	\end{equation*}
	$\eul{\B{t},0}$ is the identity map. $\eul{\B{t},k}$ sends $\B{t}^{\B{b}}$ to $\binom{-|\B{b}|-r}{k}\B{t}^{\B{b}}$. Actually, by observing that
	\begin{equation*}
		\ssm{\B{a}\in\nt^r:|\B{a}|=k}(-1)^k\binom{\B{a}+\B{b}}{\B{a}}=\ssm{\B{a}\in\nt^r:|\B{a}|=k}\binom{-\B{b}-\B{1}_r}{\B{a}},
	\end{equation*}
	the right hand side corresponds to the coefficient of $x^k$ in the Taylor series expansion of $\prod^r_{i=1}(1+x)^{-b_i-1}=(1+x)^{-|\B{b}|-r}$, which is precisely $\binom{-|\B{b}|-r}{k}$.
	
	\subsection{$F$-thresholds, $\nu$-invariants and $F$-jumping exponents}\label{subsec:Fthres}
	We fix the following notation. Let $I=\ideala{f_1,\dots,f_m}$ be a non-zero ideal of $R$. For every integer  $e\in\nt$, we define the following ideals associated to $I$.
	\begin{itemize}
		\item $I\ppm{e}=\ideala{f^{p^e}:f\in I}=\ideala{f_1^{p^e},\dots,f_r^{p^e}}$.
		\item $I\ppdm{e}$ is the unique smallest ideal $J$ of $R$ with respect to inclusion and such that $I\subseteq J\ppm{e}$.
	\end{itemize}

	\begin{example}[\text{[\textcolor{blue}{BMS09}, Proposition 2.5]}]\label{exa:BMS09}
		Suppose that $R$ is free over $R^p$. Immediately, $R$ is free over $R^{p^e}$. Let $\{u_i\}_{i=1}^n$ be the $R^{p^e}$-basis of $R$. For an ideal $I=\ideala{f_1,\dots,f_m}$ of $R$ with $f_i=\sum^m_{i=1}g_{ij}^{p^e}u_j$, we have
		\begin{equation*}
			I\ppdm{e}=\{g_{ij}:1\leq i\leq m,1\leq j\leq n\}.
		\end{equation*}
		
		In this case, $\dle{R}$ is an $R$-free algebra generated by $\{u^*_i\}_{i=1}^n$, where $u^*_i$ is given on the basis by $u^*_i(u_j)=\delta\ud{i,j}$ (the Kronecker symbol). It is clear that for all $e\in\nt$, there is
		\begin{equation*}
			\dle{R}\cdot I=\{\phi(f):\phi\in\dle{R},f\in I\}=(I\ppdm{e})\ppm{e}.
		\end{equation*}
	\end{example}
	
	\begin{definition}\label{def:nu-invar}
		Let $\mf{a}$ be an ideal of $R$ containing $I$ in its radical. For every integer $e\in\nt$, we define
		\begin{equation*}
			\nJI{e}=\max\{n\in\nt:I^n\not\subseteq \mf{a}\ppm{e}\}.
		\end{equation*}
		And we call the set $\ninv{I}(p^e)=\{\nJI{e}:R\supsetneqq \sqrt{\mf{a}}\supseteq I\}$ $\nu$-invariants of level $e$ for $I$.
	\end{definition}
	
	\begin{lemma}\label{lem:vin^I_J Noe}
		For ideals $R\supsetneq \sqrt{\mf{a}}\supseteq I$, the sequence $\{\frac{\nJI{e}}{\pow{e}}\}_{e\in\nt}$ is bounded. And we have the inequality
		\begin{equation}\label{eq:v(p^e)/p^e}
			\frac{\nJI{e_1+e_2}}{\pow{e_1+e_2}}-\frac{\nJI{e_1}}{\pow{e_1}}\leq\frac{\mu(I)}{p^{e_1}},\quad \text{for all}\;e_1,e_2\in\nt.
		\end{equation}
		Here $\mu(I)$ is the minimum of the number of generators of $I$.
	\end{lemma}

	\begin{proof}
		It is immediate to see that $\mf{a}^{(\pow{e}-1)\mu(\mf{a})+1}\subseteq \mf{a}\ppm{e}$. Suppose that $I^N\subseteq \mf{a}$, we obtain
		\begin{equation*}
			\frac{\nJI{e}}{\pow{e}}\leq \frac{N((\pow{e}-1)\mu(\mf{a})+1)}{\pow{e}}<N(\mu(\mf{a})+1).
		\end{equation*}
		\eqref{eq:v(p^e)/p^e} follows from \cite[Lemma 3.3]{Ste18}. If $I=\ideala{f}$ is a principal ideal, we can express \eqref{eq:v(p^e)/p^e} by saying that
		if $f^{\nu^{\mf{a}}_{\ideala{f}}(p^e)+1}\in \mf{a}\ppm{e}$, then $ f^{p(\nu^{\mf{a}}_{\ideala{f}}(p^e)+1)}\in \mf{a}\ppm{e+1}$.
	\end{proof}
	
	\eqref{eq:v(p^e)/p^e} implies that the limit
	\begin{equation*}
		\cpt{\mf{a}}(I)=\lim\ud{e\to\infty}\frac{\nJI{e}}{p^e}
	\end{equation*}
	exists, see \cite[Theorem 3.4]{Ste18}, and we call it the \Nft of $I$ with respect to $\mf{a}$. For simplicity, if $I=\ideala{f}$ is a principal ideal, $\nu^\mf{a}_{\ideala{f}}(p^e),\cpt{\mf{a}}(\ideala{f})$ will be denoted by $\nu^{\mf{a}}_f(p^e)$ and $\cpt{\mf{a}}(f)$, respectively.
	
	For properties that involve the $F$-thresholds, from now on we suppose that $R$ is a regular ring. By Kunz's theorem \cite[Corollary 2.7]{Kun69}, the regularity condition implies that the $e$-th iterated Frobenius map $F^e$ is flat (or $R$ is $R^{p^e}$-flat) for some (or equivalently, all) $e\in\nt$.

	\begin{lemma}[\cite{BMS08},\text{[\textcolor{blue}{BMS09}, Lemma 2.2]}]\label{lem:regular prop}
		For any element $f\in R$ and ideals $I,J$ of $R$, one verifies that: \begin{enumerate}
			\item $f\ppm{e}\in I\ppm{e}$ if and only if $f\in I$. $I\ppm{e}\subseteq J\ppm{e}$ if and only if $I\subseteq J$.
			\item $\dle{R}\cdot I=(I\ppdm{e})\ppm{e}$.
		\end{enumerate}
	\end{lemma}

	\begin{lemma}\label{lem:F-thre in reg}
		Under the assumptions above, for ideals $R\supsetneqq \sqrt{\mf{a}}\supseteq I$, the sequence $\{\frac{\nJI{e}}{p^e}\}$ is non-decreasing. Moreover, if $I\cap R^\circ\neq\varnothing$ ($R^\circ$ is the set of elements of $R$ which are not contained in any minimal prime ideals of $R$), then we have the following properties for all integers $e\in\nt$.
		\begin{enumerate}
			\item $\frac{\nJI{e}}{p^e}<\cpt{\mf{a}}(I)$.
			\item For any $f\in\sqrt{\mf{a}}$, $\nu^{\mf{a}}_f(p^e)=\lceil p^e\cpt{\mf{a}}(f)\rceil-1$. Here $\ceil{x}$ is the smallest integer $\geq x$.
		\end{enumerate}
	\end{lemma}
	
	\begin{proof}
		We have $\nJI{e+1}\geq p\nJI{e}$. The inequality is a consequence of Lemma \ref{lem:regular prop}, (1). Hence
		\begin{equation}\label{eq:nuin nonde}
			\frac{\nJI{e}}{p^e}\leq\frac{\nJI{e+1}}{p^{e+1}}\quad\text{and}\quad\frac{\nJI{e}}{p^e}\leq \cpt{\mf{a}}(I)
		\end{equation}
		as claimed. To prove (1), if $p^{e_0}\cpt{J}(I)=\nJI{e_0}$ for some $e_0$, then by \eqref{eq:nuin nonde}, for all $e\in \nt$, there is
		\begin{equation*}
			\nJI{e_0+e}=p^e\nJI{e_0}\quad\text{and}\quad I(I^{\nJI{e_0}})^{p^e}\subseteq \mf{a}\ppm{e+e_0}=(\mf{a}\ppm{e_0})\ppm{e}.
		\end{equation*}
		Choose $u\in I\cap R^\circ$, for any $g\in I^{\nJI{e_0}}$, $ug^{p^{e}}\in(\mf{a}\ppm{e_0})\ppm{e}$ for all $e\in\nt$. This leads to $g\in \mf{a}\ppm{e_0}$. Hence
		\begin{equation*}
			I^{\nJI{e_0}}\subseteq \mf{a}\ppm{e_0}.
		\end{equation*}
		A contradiction, since every ideal in a regular ring of characteristic $p$ is tighted closed (\cite{HH90}). Then for any $f\in \sqrt{\mf{a}}$, Lemma \ref{lem:vin^I_J Noe} and (1) say that for all $e\in\nt$, there is
		\begin{equation*}
			\frac{\nu^{\mf{a}}_f(p^e)}{p^e}<\cpt{\mf{a}}(f)\leq \frac{\nu^{\mf{a}}_f(p^e)+1}{p^e}.
		\end{equation*}
		This finishes the proof of (2).
	\end{proof}
	
	\begin{definition}\label{def:test ideals}
		For an ideal $I$ of $R$ and a non-negative real number $c\in\mb{R}_{\geq 0}$, the generalized test ideal of $I$ with exponent $c$ is defined as
		\begin{equation*}
			\tau(I,c)=\bcup{e\in\nt}(I^{\ceil{cp^e}})\ppdm{e}=(I^{\ceil{cp^{e_0}}})\ppdm{e_0}\;\text{for all sufficiently large}\;e_0.
		\end{equation*}
		The second equality is based on the fact that the ascending chain
		\begin{equation*}
			\cdots\subseteq(I^{\ceil{cp^e}})\ppdm{e}\subseteq(I^{\ceil{cp^{e+1}}})\ppdm{e+1}\subseteq\cdots
		\end{equation*}
		is stationary, see \cite[Lemma 2.8]{BMS08}.
	\end{definition}
	
	\begin{definition}\label{def:Fjp}
		The family of test ideals $\{\tau(I,c)\}_{c\in\mb{R}_{\geq 0}}$ is right continuous (\cite[Corollary 2.16]{BMS08}). For any $c\geq 0$, there exists an $\varepsilon>0$ such that $\tau(I,c)=\tau(I,c')$ for all $c'\in [c,c+\varepsilon)$. $\lambda\geq 0$ is called an $F$-jumping exponent of $I$ if $\tau(I,\lambda-\varepsilon)\supsetneqq\tau(I,\lambda)$ for all $\varepsilon\in(0,\lambda]$.
	\end{definition}
	
	The set of all $F$-jumping exponents of $I$ is denoted by $\textbf{FJ}(I)$. $0$ is always an $F$-jumping exponent. $\textbf{FJ}(I)$ is a discrete subset contained in $\Q$, see \cite[Theorem 3.1]{BMS08}. When $I=\ideala{f}$ is a principal ideal, we denote $\tau(\ideala{f},c)$ and $\textbf{FJ}(\ideala{f})$ by $\tau(f,c)$ and $\textbf{FJ}(f)$, respectively.
	
	\begin{lemma}[\text{[\textcolor{blue}{BMS08}, Corollary 2.30], [\textcolor{blue}{QG21b}, Proposition 4.2]}]\label{lem:FJ and Fthe}
		Fix an ideal $I$ of $R$.
		\begin{enumerate}
			\item $\nu^\bullet_I(p^e)$, $\nu$-invarints of level $e$ for $I$, is equal to $\mc{B}^\bullet_I(p^e)=\{n\in\nt:\dle{R}\cdot I^n\supsetneqq\dle{R}\cdot I^{n+1}\}$. We call $\mc{B}^\bullet_I(p^e)$ the set of different jumps of level $e$ for $I$.
			\item $\textbf{FJ}(I)$, the set of $F$-jumping exponents of $I$, is equal to $\{\cpt{\mf{a}}(I):R\supsetneqq\sqrt{\mf{a}}\supseteq I\}$.
		\end{enumerate}
	\end{lemma}

	\section{Continuous functions on $\zp$}\label{sec:cont zp}
	
	The concept of the algebra of continuous functions on $\zp$ (or $\zp^r$) is a perfect tool to discribe the \Nbs functional equations in positive characteristic.
	Let $\zp$ be the ring of $p$-adic integers with the $p$-adic metric $|\cdot|_p$. Fix an integer $e\in\nt$, for a $p$-adic integer $\mc{x}=\ssm{i\in\nt}\mc{x}^{[i]}p^i$, we call $\tre{\mc{x}}=\sum^{e-1}_{i=0}\mc{x}^{[i]}p^i$ the $e$-th truncation of $\mc{x}$. Every integer $a\in\nt$ can be uniquely written as a finite sum $a=\sum_{i\in\nt}a^{[i]}p^i$ with $a^{[i]}\in\{0,\dots,p-1\}$ and $a^{[i]}=0$ for all sufficiently large $i$.
	
	\begin{definition}
		Let $T$ be an $\fp$-algebra. We equip $T$ with the discrete topology. Let $\cut{T}$ be the set of\\
		continuous maps from $\zp$ to $T$. For any $e\in\nt$, let $\ceut{T}$ be the subset of $\cut{T}$ that consists of those maps $\Phi$, where
		\begin{equation*}
			\Phi(\mc{x})=\Phi(\mc{y}),\quad\text{whenever}\;\mc{x}-\mc{y}\in p^e\zp.
		\end{equation*}
		$\ceut{T}$ and $\cut{T}$ admit the $T$-algebra structures by pointwise addition and multiplication.
	\end{definition}
	
	\begin{remark}
		Note that $T$ is discrete, $\zp$ is a compact metrizable space, so any continuous map $\Phi\in\cut{T}$ is uniformly locally constant. There exists an $\varepsilon>0$ such that 
		\begin{equation*}
			\Phi(\mc{x})=\Phi(\mc{y}),\quad\text{whenever}\;|\mc{x}-\mc{y}|_p<\varepsilon.
		\end{equation*}
		Then $\Phi\in\ceut{T}$ when $e>-\log_p\varepsilon$. This implies that $\cut{T}=\cup\ud{e\in\nt}\ceut{T}$.
	\end{remark}
	
	We will need the following facts.
	\begin{enumerate}
		\item $T\cong\textbf{C}^0(\zp,T)$, to each element of $T$ we associate the corresponding constant map.
		\item For any $\fp$-subalgebra $S\subseteq T$, $\cut{S}$ can be naturally viewed as an $\fp$-subalgebra of $\cut{T}$.
		\item For any $e\in\nt$, $\ceut{\fp}\ott{\fp}T\cong\ceut{T}$. Every map in $\ceut{T}$ is an $\fp$-linear combination of the maps of the form $\Z{F}\cdot v$, where $\Z{F}\in\ceut{\fp}\subseteq\ceut{T}$ and $v\in T=\textbf{C}^0(\zp,T)$.
	\end{enumerate}
	
	\begin{lemma}\label{lem:basis of cfp}
		Fix an integer $e\in\nt$. For $k\in \textbf{N}_1(e)$, i.e., $k=\{0,1,\dots,p^e-1\}$, we define the functions 
		\begin{equation*}
			\Z{B}_k:\zp\to\fp,\quad \mc{x}\mapsto\binom{\tre{\mc{x}}}{k}\bmod p.
		\end{equation*}
		They make up an $\fp$-basis of the $\fp$-vector space $\ceut{\fp}$.
	\end{lemma}
	
	\begin{proof}
		Since $\z$ is dense in $\zp$, there is a one-to-one correspondence between $\ceut{\fp}$ and the $p^e$-periodic
		functions from $\z$ to $\fp$. It can be expressed as follows.
		\begin{itemize}
			\item For any $\Z{F}\in\ceut{\fp}$. The restriction of $\Z{F}\in\ceut{\fp}$ on $\z$ is a $p^e$-periodic function.
			\item Every $p^e$-periodic function from $\z$ to $\fp$ can be uniquely extended to a continuous function in $\ceut{\fp}$. 
		\end{itemize}
		
		Note that the restriction of $\Z{B}_k$ on $\z$ is the binomial polynomial $\binom{\bullet}{k}:a\in\z\mapsto \binom{a}{k}\bmod p$. They make up an $\fp$-basis of the space of $p^e$-periodic map from $\z$ to $\fp$ \cite{gtm198}. We obtain the desired conclusion.
	\end{proof}

	Suppose that $a=\sum^n_{i=0}a^{[i]}p^i$ and $k=\sum^{e-1}_{i=0}k^{[i]}p^i$. We set $k^{[i]}=0$ when $i\geq e$, then by Lucas's theorem,
	\begin{equation}\label{eq:luc}
		\binom{a}{k}\equiv\prod^{l}_{i=0}\binom{a^{[i]}}{k^{[i]}}\bmod p.
	\end{equation}
	So for each $j\in\nt$, $\Z{B}_{p^j}(\ssm{i\in\nt}\mc{x}^{[i]}p^i)=\mc{x}^{[j]}\bmod p$. Since $\nt$ is also dense in $\zp$, the following identity holds.
	\begin{equation*}
		\Z{B}_k=\prod\ud{k^{[i]}\neq 0}\frac{1}{k^{[i]}!}\,\Z{B}_{p^i}(\Z{B}_{p^i}-1)\cdots(\Z{B}_{p^i}-k^{[i]}+1).
	\end{equation*}
	Because both sides take the same values on $\nt$ by \eqref{eq:luc}.
	Here we view $k^{[i]}$ as elements in $\fp$. Hence $\ceut{\fp}$ is an $\fp$-algebra generated by $\{\Z{B}_{p^i}\}_{i=0}^{e-1}$ for all $e\in\nt$.
	
	Bitoun \cite[Theorem 1.1.8]{Bit18} proved that there is a bijection from the $p$-adic integers to the set of maximal ideals $\spmc \cut{\fp}$ of $\cut{\fp}$ given by
	\begin{equation*}
		\mc{x}\mapsto \mc{I}(\mc{x})=\{\Z{F}\in\cut{\fp}:F(\mc{x})=0\}=\ideala{\Z{B}_{p^i}-\mc{x}^{[i]}:i\in\nt}.
	\end{equation*}
	Moreover, $\cut{\fp}$ is a Jacobson ring. Each ideal $\mf{P}$ of $\cut{\fp}$ is characterized by $\bv{\mf{P}}$, the set of roots of $\mf{P}$ defined by $\{\mc{x}\in\zp:\Z{F}(\mc{x})=0,\;\text{for all}\;\Z{F}\in\mf{P}\}$. More precisely, 
	\begin{equation}\label{eq:nulls}
		\mf{P}=\bcap{\mc{x}\in \bv{\mf{P}}}\mc{I}(\mc{x}).
	\end{equation}
	
	\subsection{The $\cut{\fp}$-modules}
	
	Let $M$ be a $\cut{\fp}$-module. Given a $p$-adic integer $\mc{x}$, define $M\kud{\mc{x}}=\{u\in M:\mc{I}(\mc{x})\cdot u=0\}$ and $M_{\mc{x}}$ to be the quotient $M/\mc{I}(\mc{x})M$. For every $\mc{x}\notin\bv{\mr{ann}_{\cut{\fp}}\,M}$, there is an $\Z{F}\in\mr{ann}_{\cut{\fp}}\,M$ 
	which does not vanish at $\mc{x}$. Hence $M=(\Z{F}(\mc{x}))^{-1}(\Z{F}-\Z{F}(\mc{x}))\cdot M=\mc{I}(\mc{x})M$, $M_{\mc{x}}=0$. The following lemma is useful.
	
	\begin{lemma}[\text{[\textcolor{blue}{JNnBQG23}, Lemma 2.21]}]\label{lem:Mx=0}
		If $M_{\mc{x}}=0$ for all $\mc{x}\in\zp$, then $M=0$.
	\end{lemma}

	\begin{proposition}\label{pro: zero sets}
		Assume that there are only finitely many $p$-adic integers $\{\mc{x}_1,\dots,\mc{x}_n\}$ such that $M_{\mc{x}_i}\neq 0$.
		\begin{enumerate}
			\item $\bv{\mr{ann}_{\cut{\fp}}\,M}=\{\mc{x}_1,\dots,\mc{x}_n\}$.
			\item We have a decomposition $M=\bigoplus^n_{i=1}M\kud{\mc{x}_i}$ and $M\kud{\mc{x}_i}\cong M_{\mc{x}_i}$.
		\end{enumerate}
	\end{proposition}
	
	\begin{proof}
		We already have $\bv{\ann{\cut{\fp}}\,M}\supseteq\{\mc{x}_1,\dots,\mc{x}_n\}$. To prove (1), we only need to show that 
		\begin{equation}\label{eq:m(x) and ann}
			\bigcap^n_{i=1}\mc{I}(\mc{x}_i)\subseteq\mr{ann}_{\cut{\fp}}\,M.
		\end{equation}
		
		Take any $\Z{F}\in\cut{\fp}$ vanishing on all $\mc{x}_i$. For those $\mc{x}\notin\{\mc{x}_1,\dots,\mc{x}_n\}$, $M=\mc{I}(\mc{x})M$, so
		\begin{equation*}
			\Z{F}\cdot M=\mc{I}(\mc{x})(\Z{F}\cdot M)\quad\text{and}\quad(\Z{F}\cdot M)_{\mc{x}}=0.
		\end{equation*}
		Note that $\Z{F}\in\mc{I}(\mc{x}_i)$, $\Z{F}\cdot M=\Z{F}^p\cdot M=(\Z{F}^{p-1})\cdot (\Z{F}\cdot M)\subseteq \mc{I}(\mc{x}_i)(\Z{F}\cdot M)$. Hence $(\Z{F}\cdot M)_{\mc{x}}=0$ for all $\mc{x}\in\zp$. By Lemma \ref{lem:Mx=0}, the
		$\cut{\fp}$-module $\Z{F}\cdot M=0$, thus $\Z{F}\in\ann{\cut{\fp}}\,M$. Based on \eqref{eq:m(x) and ann}, there is
		\begin{equation*}
			\{\mc{x}_1,\dots,\mc{x}_n\}=\mb{V}\left(\bigcap^n_{i=1}\mc{I}(\mc{x}_i)\right)\supseteq \bv{\ann{\cut{\fp}}\,M}.
		\end{equation*}
		The first equality holds because for any $\mc{x}\notin\{\mc{x}_1,\dots,\mc{x}_n\}$, choose an integer $e\in\nt$ such that $\tre{\mc{x}}\neq\tre{\mc{x}_i}$ for all $i$. Let $\Z{L}^e_{\mc{x}}\in\ceut{\fp}$ be the function sending $\mc{y}$ to 1 whenever $\tre{\mc{y}}=\tre{\mc{x}}$ and to 0 otherwise. Then $\Z{L}^e_{\mc{x}}$ vanishes on all $\mc{x}_i$ but $\Z{L}^e_{\mc{x}}(\mc{x})=1$. So $\mc{x}\notin\bv{\cap^n_{i=1}\mc{I}(\mc{x}_i)}$.
		
		To prove (2), choose an integer $e\in\nt$ such that $\{\tre{\mc{x}_i}\}_{i=1}^n\in \textbf{N}_1(e)$ are different from each other. Note that $\Z{L}^e_{\mc{x}_i}$ vanishes on $\mc{x}_j$ for all $j\neq i$. Functions of $\mc{I}(\mc{x}_i)\cdot \Z{L}^e_{\mc{x}_i}$ vanish on all $\mc{x}_j$, so $\mc{I}(\mc{x}_i)\cdot \Z{L}^e_{\mc{x}_i}\subseteq\mr{ann}_{\cut{\fp}}\,M$. 
		We obtain a canonical homomorphism
		\begin{equation*}
			M_{\mc{x}_i}\to M\kud{\mc{x}_i},\quad u+\mc{I}(\mc{x}_i)M\mapsto \Z{L}^e_{\mc{x}_i}\cdot u.
		\end{equation*}
		It is injective, since $\Z{L}^e_{\mc{x}_i}-1\in\mc{I}(\mc{x}_i)$, if $\Z{L}^e_{\mc{x}_i}\cdot u\in\mc{I}(\mc{x}_i)M$, then $u=-(\Z{L}^e_{\mc{x}_i}-1)\cdot u+\Z{L}^e_{\mc{x}_i}\cdot u\in\mc{I}(\mc{x}_i)M$. And it is also surjective, for any $u\in M\kud{\mc{x}_i}$, $\Z{L}^e_{\mc{x}_i}\cdot \Z{L}^e_{\mc{x}_i}\cdot u=\Z{L}^e_{\mc{x}_i}\cdot u=u$. Hence $M\kud{\mc{x}_i}\cong M_{\mc{x}_i}$. 
		
		The decomposition $M=\bigoplus^n_{i=1}M\kud{\mc{x}_i}$ is given by $u=\sum^n_{i=1}\Z{L}^e_{\mc{x}_i}\cdot u$. It is the composition of the canonical isomorphism
		\begin{equation*}
			M=\frac{M}{(\ann{\cut{\fp}}\,M)M}=\frac{M}{(\bigcap^n_{i=1}\mc{I}(\mc{x}_i))M}\cong\bigoplus^n_{i=1}M_{\mc{x}_i}
		\end{equation*}
		and the isomorphisms $M_{\mc{x}_i}\cong M\kud{\mc{x}_i}$ we given above. 
	\end{proof}

	\begin{corollary}\label{cor:week decom}
		Assume that there are $p$-adic integers $\{\mc{x}_1,\dots,\mc{x}_n\}$ such that $M$ splits as
		\begin{equation*}
			M=\bigoplus^n_{i=1}M'_i,\quad\text{with}\;0\neq M'_i\subseteq M\kud{\mc{x}_i}.
		\end{equation*}
		Then $M'_i=M\kud{\mc{x}_i}\cong M_{\mc{x}_i}$.
	\end{corollary}
	
	\begin{proof}
		Choose an integer $e\in\nt$ such that $\{\tre{\mc{x}_i}\}_{i=1}^n$ are different from each other. When $j\neq i$, $-\Z{L}^e_{\mc{x}_j}+1\in\mc{I}(\mc{x}_j)$, $\Z{L}^e_{\mc{x}_j}\in\mc{I}(\mc{x}_i)$, thus $\mc{I}(\mc{x}_j)M'_i=M'_i$. Clearly $\mc{I}(\mc{x}_i)M'_i=0$. Hence $M_{\mc{x}_i}\cong M'_i$. 
		
		For those $\mc{x}\notin\{\mc{x}_1,\dots,\mc{x}_n\}$, using the same methods above, we see that $\mc{I}(\mc{x})M'_i=M'_i$ for all $i$, $M_{\mc{x}}=0$. The set of $p$-adic number $\mc{x}$ such that $M_{\mc{x}}\neq 0$ is just $\{\mc{x}_1,\dots,\mc{x}_n\}$. Then we apply Proposition \ref{pro: zero sets}.
	\end{proof}
	
	\subsection{Continuous functions on $\zp^r$}
	In the same way, fix an integer $r\in\nz$. Let $T$ be an $\fp$-algebra. We equip $T$ with the discrete topology and $\zp^r$ with the product topology. The $T$-algebra of continuous maps from $\zp^r$ to $T$ is denoted by $\rcut{T}$. For any $e\in\nt$, let $\rceut{T}$ be the $T$-subalgebra of $\rcut{T}$ that consists of those continuous maps $\Phi$ from $\zp^r$ to $T$ such that
	\begin{equation*}
		\Phi(\mc{x}_1,\dots,\mc{x}_r)=\Phi(\mc{y}_1,\dots,\mc{y}_r),\quad\text{whenever}\;(\mc{x}_1,\dots,\mc{x}_r)-(\mc{y}_1,\dots,\mc{y}_r)\in p^e\zp^r.
	\end{equation*}
	Similarly, we have $\rcut{T}=\bcup{e\in\nt}\rceut{T}$, and $\rceut{T}=\rceut{\fp}\ott{\fp}T$ for each $e\in\nt$. The facts we mentioned at the beginning of this section still hold for general $r\in\nz$.
	
	Fix an integer $e\in\nt$, the functions
	\begin{equation}\label{eq:functions of Bk}
		\Z{B}_{\B{k}}:\zp^r\to\fp,\quad (\mc{x}_1,\dots,\mc{x}_r)\mapsto\prod^r_{i=1}\binom{\tre{\mc{x}_i}}{k_i}\bmod p,\quad \B{k}=(k_1,\dots,k_r)\in N_r(e),
	\end{equation}
	make up an $\fp$-basis of the $\fp$-vector space $\rceut{\fp}$. Another $\fp$-basis of $\rceut{\fp}$ is given by
	\begin{equation*}
		\Z{L}^e_{\B{k}}:\zp^r\to \fp, (\mc{x}_1,\dots,\mc{x}_r)\mapsto\begin{cases}
			1,&\text{if}\;(\tre{\mc{x}_1},\dots\tre{\mc{x}_r})=\B{k},\\
			0,&\text{otherwise}.
		\end{cases}\quad \B{k}\in N_r(e).
	\end{equation*}
	
	\begin{lemma}\label{lem:Cfp to DR}
		Let $R$ be an $F$-finite ring and $R[\B{t}]$ be the polynomial ring $R[t_1,\dots,t_r]$. We have an injective
		$D_R$-algebra homomorphism
		\begin{equation*}
			\Delta\ud{r,R}:\rcut{D_R}\to (D_{R[\B{t}]})_0,\quad \Delta\ud{r,R}(\Z{P}):u\B{t}^{\B{a}}\mapsto\Z{P}(-\B{a}-\B{1}_r)(u)\B{t}^{\B{a}}.
		\end{equation*}
		Here we note that $\Z{P}\in\rcut{D_R}$, so $\Z{P}(-\B{a}-\B{1}_r)$ is a \Ndpo on $R$.
	\end{lemma}
	
	\begin{proof}
		It is easy to check that $\Delta\ud{r,R}$ is a $D_R$-algebra homomorphism. There is $\Delta\ud{r,R}(\ceut{\dle{R}})\subseteq\dle{R[\B{t}]}$. Now we prove that it is injective. For any $\Z{P}\in\knr\Delta\ud{r,R}$, assume that $\Z{P}\in\rceut{D_R}$, then $\Z{P}(\B{b})=0$ when $\B{b}\in\{-\B{1}_r-\B{a}:\B{a}\in\nt^r\}+p^e\zp^r\supseteq\nt^r$. $\Z{P}$ vanishes on a dense subset of $\zp^r$, thus $\Z{P}=0$.  
	\end{proof}
	
	For every $\B{k}\in N_r(e)$ and $\phi\in D_R$, $\Delta\ud{r,R}(\Z{B}_{\B{k}}\cdot\phi)=\phi\,(-1)^{|\B{k}|}\dpot{\B{k}}\B{t}^{\B{k}}$, and
	\begin{equation}\label{eq:action of Lek}
		\Delta\ud{r,R}(\Z{L}^e_{\B{k}}\cdot\phi):u\B{t}^{\B{a}}\mapsto\begin{cases}
			\phi(u)\B{t}^{\B{a}},&\text{if}\;\B{a}\in(p^e-1)\B{1}_r-\B{k}+p^e\nt^r,\\
			0,&\text{otherwise}.
		\end{cases}
	\end{equation}
	Composed with the injective $D_R$-algebra homomorphism
	\begin{equation}\label{eq:Sigma}
		\Sigma\ud{r,R}:\cut{D_R}\to\rcut{D_R},\quad \Sigma\ud{r,R}(\Z{F}):(\mc{x}_1,\dots,\mc{x}_r)\mapsto\Z{F}(\mc{x}_1+\cdots+\mc{x}_r).
	\end{equation}
	We obtain an injective $D_R$-algebra homomorphism $\Delta'\ud{r,R}:\cut{D_R}\to(D_{R[\B{t}]})_0$ sending $\Z{B}_k$ to Euler operator $\vartheta\ud{\B{t},k}$. This means that any $(D_{R[\B{t}]})_0$-module $M$ can be viewed as a $\cut{\fp}$-module through $\Delta'_{r,R}$ (Here we use the fact that $\cut{\fp}\subseteq\cut{D_R}$).
	
	\section{Bernstein-Sato theory in positive characteristic}\label{sec: BS theory}
	In this section ,we let $R$ be an $F$-finite ring of prime characteristic $p$, and let $I$ be a non-zero ideal of $R$. Choose a system of generators $\mc{F}=\{f_i\}_{i=1}^r$ of $I$. Suppose that $\B{f}=f_1\cdots f_r$ is a non-zero divisor of $R$. 
	
	\subsection{Local cohomomology construction of Bernstein-Sato roots}
	
	\begin{definition}
		A $p$-adic integer $\mc{x}\in\zp$ is a \Nbs root of $I$ if there exists a sequence $\{\nu\ud{e}\}_{e\in\nt}$ 
		with $\nu\ud{e}\in\mc{B}^\bullet_I(p^e)=\{n\in\nt:\dle{R}\cdot I^n\supsetneqq\dle{R}\cdot I^{n+1}\}$ such that $\mc{x}$ is a $p$-adic limit of $\{\nu\ud{e}\}_{e\in\nt}$. $\textbf{BSR}(I)$ denotes the set of all \Nbs roots of $I$. If $I=\ideala{f}$ is a principal ideal, we write $\textbf{BSR}(f)$ instead of $\textbf{BSR}(\ideala{f})$.
	\end{definition}
	
	One of the obvious problems of our definition of \Nbs roots is that it is hard to compute these roots and study their properties. To deal with this problem, Jeffries, N\'{u}\~{n}ez-Betancourt and
	Quinlan-Gallego provided the following results.
	We first recall the following notations introduced in Section \ref{sec:intro}.
	
	\begin{Setup}
		Consider the polynomial ring $R[\B{t}]=R[t_1,\dots,t_r]$. Let $\ideala{\B{t}}=\ideala{t_1,\dots,t_r}$ and $J=\ideala{f_i-t_i:1\leq i\leq r}$ be ideals of $R[\B{t}]$. The local cohomology module $H_{\mc{F}}=H^r_JR[\B{t}]$ (see \eqref{eq:local coho}) is a free $R[\B{t}]$-module with the basis
		\begin{equation*}
			\delta\ud{\B{k}}=\left[\frac{1}{\prod^r_{i=1}(f_i-t_i)^{k_i}}\right],\quad\B{k}=(k_1,\dots,k_r)\in(\nz)^r.
		\end{equation*}
		$H_{\mc{F}}$ carries a $D_{R[\B{t}]}$-module structure. For any $\varphi\in\dlr{R[\B{t}]}$ and $\B{k}\in(\nz)^r$, $\varphi\cdot\delta\ud{\B{k}}=\varphi(\prod^r_{i=1}(f_i-t_i)^{(p^e-1)k_i})\delta\ud{p^e\B{k}}$.\\
		Let $\delta=\delta\ud{\B{1}_r}$. Then $f_i\delta=t_i\delta$. Let $N_{\mc{F}}$ be the $(D_{R[\B{t}]})_0$-module
		\begin{equation}\label{eq:N_I}
			N_{\mc{F}}=\frac{V^0D_{R[\B{t}]}\cdot \delta}{V^1D_{R[\B{t}]}\cdot \delta}=\frac{(D_{R[\B{t}]})_0\cdot\delta}{(D_{R[\B{t}]})_1\cdot\delta},
		\end{equation}	
		where the second equality comes from Lemma \ref{eq:Vfil and deg} and the fact that when $j>i$,
		\begin{equation*}
			(D_{R[\B{t}]})_j\cdot\delta=(D_{R[\B{t}]})_i\cdot\ideala{\B{t}}^{j-i}\cdot \delta=(D_{R[\B{t}]})_i\cdot I^{j-i}\cdot \delta\subseteq (D_{R[\B{t}]})_i\cdot\delta.
		\end{equation*}
		$N_{\mc{F}}$ defined above depends on the choice of generators.
	\end{Setup}
	
	\begin{Setup}
		As $f_i$ is invertible in $R_{\B{f}}$, for each $\B{a}=(a_1,\dots,a_r)\in\z^r$, we let $\B{f}^{\B{a}}=\prod\ud{a_i<0}(f_i^{-1})^{-a_i}\prod\ud{a_i\geq 0}f_i^{a_i}\in R_{\B{f}}$.
		Let $J'$ be the ideal $JR_{\B{f}}[\B{t}]$ of $R_{\B{f}}[\B{t}]=R_{\B{f}}[t_1,\dots,t_r]$, and let $H'_{\mc{F}}$ denote the local cohomology module $H^r_{J'}R_{\B{f}}[\B{t}]$. It is a free $R_{\B{f}}[\B{t}]$-module with the basis
		\begin{equation*}
			\delta'_{\B{k}}=\left[\frac{1}{\prod^r_{i=1}(f_i-t_i)^{k_i}}\right],\quad\B{k}=(k_1,\dots,k_r)\in(\nz)^r.
		\end{equation*}
		Let $\delta'=\delta'_{\B{1}_r}$. We have a $D_{R[\B{t}]}$-linear embedding $\iota:H_{\mc{F}}\to H'_{\mc{F}}$ sending $\delta_{\B{k}}$ to $\delta'_{\B{k}}$.
		
		Let $(H'_{\mc{F}})^e$ be the free $R_{\B{f}}[t]$-submodule of $H'_{\mc{F}}$ generated by $\{\delta'_{\B{k}}:\B{k}\in(\nz)^r,\B{k}\leq p^e\B{1}_r\}$. Note that
		\begin{equation}\label{eq:decom of delta}
			\delta'_{\B{k}}=\left(\prod^r_{i=1}(f_i-t_i)^{p^e-k_i}\right)\delta'_{p^e\B{1}_r}=\ssm{\B{a}\in\nt^r,\B{a}+\B{k}\leq p^e\B{1}_r}(-1)^{|\B{a}|}\Z{B}_{\B{a}}(p^e\B{1}_r-\B{k})\,\B{f}^{\B{a}}\B{t}^{p^e\B{1}_r-\B{k}-\B{a}}\delta'_{p^e\B{1}_r},
		\end{equation}
		from the decomposition we see that $\{Q_{e,\B{k}}:=\B{t}^{(p^e-1)\B{1}_r-\B{k}}\delta'_{p^e\B{1}_r}\}_{\B{k}\in N_r(e)}$ is another $R_{\B{f}}[\B{t}]$-basis of $(H'_{\mc{F}})^e$.
	\end{Setup}

	\begin{theorem}[\text{[\textcolor{blue}{JNnBQG23}, Theorem 4.13, Theorem 4.16, Theorem 7.3]}]\label{thm:ration}
		By the homomorphism $\Delta'\ud{r,R}:\cut{D_R}\to (D_{R[\B{t}]})_0$ given in Lemma \ref{lem:Cfp to DR}, $N_\mc{F}$ admits a $\cut{\fp}$-module structure. A $p$-adic integer $\mc{x}$ is a \Nbs root of $I$ if and only if the quotient module
		\begin{equation*}
			(N_\mc{F})_{\mc{x}}=\frac{N_{\mc{F}}}{\mc{I}(\mc{x})N_\mc{F}}\neq 0,\quad\text{where}\;\mc{I}(\mc{x})=\{\Z{F}\in\cut{\fp}:\Z{F}(\mc{x})=0\}.
		\end{equation*}
		When $R$ is a regular, $I$ has finitely many \Nbs roots. All of them are negative and rational,  hence contained in $\z\kud{p}=\{\frac{a}{b}\in\Q:p\nmid b\}$.
	\end{theorem}
	
	This theorem is a characteristic $p$ analogue of Kashiwara's rationality theorem. By applying Proposition \ref{pro: zero sets}, we obtain the following corollary.
	\begin{corollary}\label{cor:Basic pro of BS-roots}
		Suppose that $R$ is a regular $F$-finite ring. $I=\ideala{f_1,\dots,f_r}$. Let $N_{\mc{F}}$ be the $\cut{\fp}$-module defined in \eqref{eq:N_I}.
		\begin{enumerate}
			\item $\textbf{BSR}(I)=\bv{\ann{\cut{\fp}}\,N_\mc{F}}$.
			\item A $p$-adic integer $\mc{x}=\ssm{i\in\nt}\mc{x}^{[i]}p^i\in\textbf{BSR}(I)$ if and only if the multi-eigenspace
			\begin{equation*}
				(N_\mc{F})\kud{\mc{x}}=\{u\in N_I:\eul{\B{t},p^i}\cdot u=\mc{x}^{[i]}u\;\text{for all}\;i\in\nt\}\neq 0.
			\end{equation*}
			\item If $\textbf{BSR}(I)=\{\mc{x}_1,\dots,\mc{x}_n\}$, then $N_{\mc{F}}$ admits a multi-eigenspace decomposition $N_\mc{F}=\bigoplus^n_{i=1}(N_\mc{F})\kud{\mc{x}_i}$.
		\end{enumerate}
	\end{corollary}
	
	Quinlan-Gallego established a close relationship between \Nbs roots and $F$-jumping exponents. 
	
	\begin{theorem}[\text{[\textcolor{blue}{QG21b}, Theorem 6.11]}]\label{thm:FJ and BSR}
		Keep the notation above. The cosets of $-\textbf{FJ}(I)\cap\z_{(p)}$ in $\z$ concides with the cosets of $\textbf{BSR}(I)$ in $\z$, i.e., $\textbf{BSR}(I)+\z=-\textbf{FJ}(I)\cap\z_{(p)}+\z$. Furthermore, for a non-zero element $f$ of $R$, $\textbf{BSR}(f)=\{-\lambda:\lambda\in\textbf{FJ}(f)\cap(0,1]\cap\z_{(p)}\}$.
	\end{theorem}
	
	\subsection{Direct construction of Bernstein-Sato functional equations}\label{subsec:Direct}
	
	In this section, we construct the Bernstein-Sato functional equations similar to \eqref{eq:BS-equ in cpx} for arbitrary ideals in an $F$-finite ring of characteristic $p$. For a single element $f\in R$, Jeffries and co-authors \cite{JNQ23} defined the $\cut{D_{R_f}}$-module $\cut{R_f}\bxf$. They gave an alternative characterization of the $(D_{R[t]})_0$-module 
	\begin{equation*}
		N_f=\frac{(D_{R[t]})_0\cdot[\frac{1}{f-t}]}{(D_{R[t]})_1\cdot[\frac{1}{f-t}]},\quad\text{with}\;[\frac{1}{f-t}]\in R[t]_{f-t}/R[t].
	\end{equation*}
	We generalize their construction. Let $R$ be an $F$-finite ring. Fix elements $\mc{F}=\{f_i\}_{i=1}^r$ of $R$. Suppose that $\B{f}=f_1\cdots f_r$ is a non-zero divisor of $R$.

	\begin{definition}
		We define the free rank one $\rcut{R_{\B{f}}}$-module $\rcut{R_{\B{f}}}\Lxf$, where $\Lxf$ is a 
		formal symbol for the generator. For every $\B{a}\in\z^r$, the exponent specialization map $E_{\B{a}}$ is given by
		\begin{equation*}
			E_{\B{a}}:\rcut{R_{\B{f}}}\Lxf\to\\R_{\B{f}},\quad \Z{U}\Lxf\mapsto\Z{U}(\B{a})\B{f}^{\B{a}},\quad\text{for all}\;\Z{U}\in\rcut{R_{\B{f}}}.
		\end{equation*}
	\end{definition}
	
	\begin{theorem}\label{thm:f1f2--}
		$\rcut{R_{\B{f}}}\Lxf$ admits a $\rcut{D_{R_{\B{f}}}}$-module structure compatible with every exponent specialization map $E_{\B{a}}$, i.e., for any $\B{a}\in\z^r$,
		\begin{equation}\label{eq:comp with exp}
			E_{\B{a}}(\Z{P}\cdot\Z{U}\Lxf)=\Z{P}(\B{a})(E_{\B{a}}(\Z{U}\Lxf)),\quad\text{for all}\; \Z{P}\in\rcut{D_{R_{\B{f}}}},\Z{U}\in\rcut{R_{\B{f}}}.
		\end{equation}
		With the $\rcut{D_{R_{\B{f}}}}$-module structure on $H'_{\mc{F}}$ induced by the homomorphism $\Delta\ud{r,R_{\B{f}}}$ in Lemma \ref{lem:Cfp to DR}, there is a $\rcut{D_{R_{\B{f}}}}$-module isomorphism $\rcut{R_{\B{f}}}\cong H'_{\mc{F}}$.
	\end{theorem}
	
	Let us first canonically endow $\rcut{R_{\B{f}}}\Lxf$ with a $\rcut{D_{R_{\B{f}}}}$-module structure. For a \Ndpo $\phi\in D_{R_{\B{f}}}$ and an element $u\in R_{\B{f}}$, the associated continuous map $\Z{V}_{\phi,u}\in\rcut{R_{\B{f}}}$ is defined by
	\begin{equation*}
		\Z{V}_{\phi,u}:(\mc{x}_1,\dots,\mc{x}_r)\mapsto\frac{\phi(uf_1^{\tre{\mc{x}_1}}\cdots f_r^{\tre{\mc{x}_r}})}{f_1^{\tre{\mc{x}_1}}\cdots f_r^{\tre{\mc{x}_r}}},\quad\text{where}\; e=\min\{l\in\nt:\phi\in \textbf{D}^{l}_R\}.
	\end{equation*}
	$\Z{V}_{\phi,u}\in\rceut{R_{\B{f}}}$. For any $r$-tuple $\B{a}=(a_1,\dots,a_r)\in\z^r$, $a_i-\tre{a_i}\in p^e\z$, we see that $\Z{V}_{\phi,u}(\B{a})\Y{f}^{\B{a}}=\phi(u\Y{f}^{\B{a}})$.

	In general, for any $\Z{P}\in\rcut{D_R}$ and $\Z{U}\in\rcut{R_{\B{f}}}$,
	we choose an integer $e\in\nt$ such that $\Z{P}$ and $\Z{U}$ can be written as
	\begin{equation*}
		\Z{P}=\sum^n_{i=1}\Z{G}_i\cdot\phi\ud{i},\Z{U}=\sum^m_{j=1}\Z{H}_j\cdot u_j,\quad\Z{G}_i,\Z{H}_j\in\rceut{\fp},\phi\ud{i}\in \dlr{R},u_j\in R_{\B{f}}.
	\end{equation*}
	Then the action of $\Z{P}$ on $\Z{U}\Lxf$ is given by
	\begin{equation}\label{eq:action on fs}
		\Z{P}\cdot \Z{U}\Lxf=\sum^n_{i=1}\sum^m_{j=1}(\Z{G}_i\Z{H}_j\cdot \Z{V}_{\phi\ud{i},u_j})\Lxf.
	\end{equation}
	From the statement above, we immediately verify that \eqref{eq:comp with exp} holds for every $\B{a}\in\z^r$.
	
	Now we show that the rule \eqref{eq:action on fs} yields a module action. It suffices to prove that for any $\Z{P}_1,\Z{P}_2\in\rcut{D_{R_{\B{f}}}}$ and $\Z{U}\in\rcut{R_{\B{f}}}$, suppose that $\Z{P}_1\cdot\Z{P}_2=\Z{P}_3$ and
	\begin{equation*}
		\Z{P}_2\cdot\Z{U}\Lxf=\Z{U}_1\Lxf,\quad\Z{P}_1\cdot\Z{U}_1\Lxf=\Z{U}_2\Lxf,\quad \Z{P}_3\cdot\Z{U}\Lxf=\Z{U}_3\Lxf,
	\end{equation*}
	where $\Z{P}_3\in\rcut{D_{R_{\B{f}}}}$ and $\Z{U}_1,\Z{U}_2,\Z{U}_3\in\rcut{R_{\B{f}}}$, there is $\Z{U}_2=\Z{U}_3$. By \eqref{eq:comp with exp}, for any $\B{a}\in\z^r$, we have
	\begin{align*}
		\Z{U}_3(\B{a})\Y{f}^{\B{a}}&=E_{\B{a}}(\Z{P}_3\cdot \Z{U}\Lxf)=\Z{P}_3(\B{a})(\Z{U}(\B{a})\Y{f}^{\B{a}})
		=(\Z{P}_1(\B{a})\cdot\Z{P}_2(\B{a}))(\Z{U}(\B{a})\Y{f}^{\B{a}})\\
		&=\Z{P}_1(\B{a})(\Z{P}_2(\B{a})(\Z{U}(\B{a})\Y{f}^{\B{a}}))=\Z{P}_1(\B{a})(\Z{U}_1(\B{a})\Y{f}^{\B{a}})=E_{\B{a}}(\Z{P}_1\cdot \Z{U}_1\Lxf)=\Z{U}_2(\B{a})\Y{f}^{\B{a}}.
	\end{align*}
	This means that $\Z{U}_2|_{\z^r}=\Z{U}_3|_{\z^r}$, thus $\Z{U}_2=\Z{U}_3$. Hence $\rcut{R_{\B{f}}}$ admits a $\rcut{D_{R_{\B{f}}}}$-module structure.
	
	Next we define the additive homomorphism $\Theta\ud{\mc{F}}:\rcut{R_{\B{f}}}\to H'_{\mc{F}}$ as follows.
	\begin{equation*}
		\Theta\ud{\mc{F}}:\Z{U}\Lxf\mapsto \Delta\ud{r,R_{\B{f}}}(\Z{U})(\delta'),\quad\text{here}\; \Z{U}\in\rcut{R_{\B{f}}}\subseteq\rcut{D_{R_{\B{f}}}}.
	\end{equation*}
	$\Delta\ud{r,R_{\B{f}}}(\rcut{R_{\B{f}}})\subseteq(D_{R_{\B{f}}[\B{t}]})_0$ is a commutative $R_{\B{f}}$-algebra. For 
	instance this shows that if we want to show that $\Theta\ud{\mc{F}}$ is $\rcut{D_{R_{\B{f}}}}$-linear, it suffices to prove that for any $\phi\in D_{R_{\B{f}}}$ and $u\in R_{\B{f}}$, there is
	\begin{equation*}
		\Theta_{\mc{F}}(\phi\cdot u\Lxf)=\Theta_{\mc{F}}(\Z{V}_{\phi,u}\Lxf)=\Delta\ud{r,R_{\B{f}}}(\Z{V}_{u,\phi})(\delta')=\phi (u\delta').
	\end{equation*}
	Suppose that $\phi\in\dle{R_{\B{f}}}$, then we have
	\begin{gather*}
		\Z{V}_{\phi,u}=\ssm{\B{a}\in N_r(e)}\frac{\phi(u\B{f}^{\B{a}})}{\B{f}^{\B{a}}}\cdot\Z{L}^e_{\B{a}}\quad\text{and}\quad \delta'=\ssm{\B{a}\in N_r(e)}(-1)^{|\B{a}|}\B{f}^{\B{a}}Q_{e,\B{a}}.
	\end{gather*}
	The second identity comes from \eqref{eq:decom of delta} and the fact that $\binom{p^e-1}{a}\equiv (-1)^a\bmod p$ for any $a\in \textbf{N}_1(e)$. Then by \eqref{eq:action of Lek}, we see that $\Delta\ud{r,R_{\B{f}}}(\Z{L}^e_{\B{k}})(\delta')=(-1)^{|\B{k}|}\B{f}^{\B{k}}Q_{e,\B{k}}$ for each $\B{k}\in N_r(e)$. Thus
	\begin{equation*}
		\Delta\ud{r,R_{\B{f}}}(V_{u,\phi})(\delta')=\ssm{\B{a}\in N_r(e)}\frac{\phi(u\B{f}^{\B{a}})}{\B{f}^{\B{a}}}(-1)^{|\B{a}|}\B{f}^{\B{a}}Q_{e,\B{a}}=\phi(u\delta').
	\end{equation*}
	This shows that $\Theta\ud{\mc{F}}$ is a $D_{R_{\B{f}}}$-module homomorphism. Furthermore, for any $e\in\nt$,
	$\{\Z{L}^e_{\B{k}}\Lxf\}_{\B{k}\in\nt^e}$ is an\\ $R_{\B{f}}$-basis of $\rceut{R_{\B{f}}}\Lxf$, their images under $\Theta\ud{\mc{F}}$ are $\{(-1)^{|\B{k}|}f^{\B{k}}Q_{e,\B{k}}\}_{\B{k}\in N_r(e)}$, which in turn form an $R_{\B{f}}$-basis of $(H'_{\mc{F}})^e$. Note that $\rcut{R_{\B{f}}}\Lxf=\cup\ud{e\in\nt}\,\rceut{R_\B{f}}\Lxf$ and $H'_{\mc{F}}=\cup\ud{e\in\nt}(H'_{\mc{F}})^e$, $\Theta\ud{\mc{F}}$ is an isomorphism.
	
	Via the $D_{R[\B{t}]}$-linear embedding $i_\mc{F}$, we want to know the inverse image  of $i_\mc{F}((D_{R[\B{t}]})_0\cdot \delta)$ and $i_\mc{F}((D_{R[\B{t}]})_0\cdot \delta)$ under $\Theta_\mc{F}$. When $r=1$, i.e., $\mc{F}=f$ for some $f\in R$, we have
	\begin{equation*}
		\Theta_\mc{F}(\cut{D_R}\cdot\bxf)=i_\mc{F}((D_{R[t]})_0\cdot\delta)\quad\text{and}\quad
		\Theta_\mc{F}(\cut{D_R}\cdot f\bxf)=i_\mc{F}((D_{R[t]})_1\cdot\delta).
	\end{equation*}
	Thus $N_f\cong \cut{D_R}\cdot\bxf/\cut{D_R}\cdot f\bxf$. $\Z{F}\in\cut{\fp}\subseteq\cut{D_R}$ satisfies a Bernstein-Sato functional equation for $f$ if there is a map $\Z{P}\in\cut{D_R}$ such that $\Z{F}\bxf=\Z{P}\cdot f\bxf$. And $\ann{\cut{\fp}}N_f$ consists of those $\Z{F}\in\cut{\fp}$ satisfying a \Nbs functional equation for $f$. 
	
	For general $r\in\nz$, we need to define some new \Ndsp on $R[\B{t}]$. 
	
	\begin{definition}
		Let $\textbf{Fun}_r$ be the set of functions $\alpha$ from $\mc{P}_r=\{1,\dots,r\}\times\nt$ to $\{-p+1,\dots,-1,0,1,\dots,p-1\}$ satisfying the condition that for each $i\in\{1,\dots,r\}$, the set $\{(i,j):\alpha(i,j)\neq 0\}$ is finite. The norm of $\alpha$, denoted by $\nm{\alpha}$, is the integer $\ssm{(i,j)\in\mc{P}_r}\alpha(i,j)p^j$. We define the operator on $R[\B{t}]$ associated to $\alpha$ as
		\begin{equation*}
			\mc{Dp}^{[\alpha]}=\prod^r_{i=1} \prod\ud{\alpha(i,j)<0}\partial^{[-\alpha(i,j)p^j]}_{t_i}\prod\ud{\alpha(i,j)>0}t_i^{\alpha(i,j)p^j}.
		\end{equation*}
		Clearly, $\mc{Dp}^{[\alpha]}$ lies in $(D_{R[\B{t}]})_{\nm{\alpha}}$.
	\end{definition}

	\begin{lemma}
		Let $n\in\nt$ be an integer. $(D_{R[\B{t}]})_n$ is a $\rcut{D_R}$-module generated by 
		\begin{equation*}
			\{\mc{Dp}^{[\alpha]}:\alpha\in\textbf{Fun}_r,\nm{\alpha}=n\}.
		\end{equation*}
	\end{lemma}
	
	\begin{proof}
		For an operator $\dpot{\B{a}}\B{t}^{\B{b}}$ such that $\B{a}=(a_1,\dots,a_r),\B{b}=(b_1,\dots,b_r)\in\nt^r$ and $|\B{b}|-|\B{a}|=n$, write each $a_i$ and $b_i$ as the $p$-adic expansions 
		\begin{equation*}
			a_i=\ssm{j\in\nt}a_i^{[j]}p^j,\quad b_i=\ssm{j\in\nt}b_i^{[j]}p^j.
		\end{equation*}
		Let $c_i=\ssm{j\in\nt}\min\{a_i^{[j]},b_i^{[j]}\}p^j$ for $i=1,\dots,r$. We define a function $\alpha\in\textbf{Fun}_r$ by setting $\alpha(i,j)=b_i^{[j]}-a_i^{[j]}$. By easy computations, based on the facts below
		\begin{equation*}
			\dpot{\B{u}}\dpot{\B{v}}=\binom{\B{u}+\B{v}}{\B{u}}\dpot{\B{u}+\B{v}},\quad [\dpot{p^e\B{u}},\dpot{\B{k}}\B{t}^{\B{k}}]=0,\quad\text{for all}\;\B{u},\B{v}\in\nt^r, \B{k}\in\textbf{N}_r(e),
		\end{equation*}
		we obtain the following conclusion.
		\begin{equation*}
			\dpot{\B{a}}\B{t}^{\B{b}}=\lambda\,\dpot{\B{c}}\B{t}^{\B{c}}\cdot\mc{Dp}^{[\alpha]}=\lambda\,\Delta\ud{r,R}(\Z{B}_{\B{c}})\cdot\mc{Dp}^{[\alpha]},\quad\text{where}\;\B{c}=(c_1,\dots,c_r),\lambda\in\fp^*.
		\end{equation*}
		Hence $(D_{R[\B{t}]})_n$ is generated by $\{\mc{Dp}^{[\alpha]}:\alpha\in\textbf{Fun}_r,\nm{\alpha}=n\}$ as a $\cut{D_R}$-module.
		
		For example, if we take $p=3,r=2,n=1$ and $\B{a}=(5,4),\B{b}=(2,8)$, then we have $\B{c}=(2,4)$ and
		\begin{align*}
			\partial^{[5]}_{t_1}\partial^{[4]}_{t_2}t_1^2t_2^8&=(\partial^{[3]}_{t_1}\partial^{[2]}_{t_1}t_1^2)(\partial^{[3]}_{t_2}\partial^{[1]}_{t_2}t_2^{6}t_2^2)=(\partial^{[3]}_{t_1}\partial^{[2]}_{t_1}t_1^2)(\partial^{[3]}_{t_2}t_2^{6}\partial^{[1]}_{t_2}t_2^2)\\
			&=(\partial^{[2]}_{t_1}t_1^2\partial^{[3]}_{t_1})(\partial^{[3]}_{t_2}t_2^3\partial^{[1]}_{t_2}t_2^1t_2^4)=\partial^{[2]}_{t_1}\partial^{[4]}_{t_2}t_1^2t_2^4\cdot\mc{Dp}^{[\alpha]},
		\end{align*}
		where $\alpha\in\textbf{Fun}_2$ satisfies $\alpha(1,1)=-1,\alpha(2,0)=\alpha(2,1)=1$, and other $\alpha(i,j)=0$.
	\end{proof}

	For any $\alpha\in\textbf{Fun}_r$, we define the function $\Z{Bd}^{[\alpha]}\in\rcut{\fp}$ as $\Z{B}_{\alpha\ud{-}}$, where 
	\begin{equation*}
		\alpha\ud{-}=\left(\sum_{\alpha(1,j)<0}-\alpha(1,j)p^j,\dots,\sum_{\alpha(r,j)<0}-\alpha(r,j)p^j\right).
	\end{equation*}
	And we define the element $\B{f}^{[\alpha]}=\prod^r_{i=1}f_i^{\sum\ud{j\in\nt}\alpha(i,j)p^j}\in R_\B{f}$. Immediately we have $\Z{Bd}^{[\alpha]}\B{f}^{[\alpha]}\cdot \delta'=\mc{Dp}^{[\alpha]}(\delta')$. Combining these results, we give another description of the $\cut{\fp}$-module $N_{\mc{F}}$.
	
	\begin{corollary}\label{cor:mthmA}
		View $\rcut{D_R}$ as an $\fp$-subalgebra of $\rcut{D_{R_{\B{f}}}}$. For each $n\in\nt$, we have
		\begin{equation*}
			\iota_{\mc{F}}((D_{R[\B{t}]})_n\cdot \delta)=(D_{R[\B{t}]})_n\cdot\delta'=\Theta\ud{\mc{F}}\left(\ssm{\alpha\in\textbf{Fun}_r,\nm{\alpha}=n}\rcut{D_R}\cdot\Z{Bd}^{[\alpha]}\B{f}^{[\alpha]}\Lxf\right).
		\end{equation*}
		Hence $N_{\mc{F}}$ is isomorphic to the $\rcut{D_R}$-module
		\begin{equation*}
			\frac{\ssm{\alpha\in\textbf{Fun}_r,\nm{\alpha}=0}\rcut{D_R}\cdot\Z{Bd}^{[\alpha]}\B{f}^{[\alpha]}\Lxf}{\ssm{\alpha\in\textbf{Fun}_r,\nm{\alpha}=1}\rcut{D_R}\cdot\Z{Bd}^{[\alpha]}\B{f}^{[\alpha]}\Lxf}.
		\end{equation*}
		With the $\cut{\fp}$-module structure on $N_\mc{F}$ induced by the homomorphism $\Sigma'\ud{r,R}$ in \eqref{eq:Sigma}, any fucntion $\Z{F}\in\\\cut{\fp}$ annihilates $N_\mc{F}$ if and only if $\Sigma'\ud{r,R}(\Z{F})\cdot \delta\subseteq (D_{R[\B{t}]})_1\cdot\delta$, as $(D_{R[\B{t}]})_0\cdot (D_{R[\B{t}]})_1=(D_{R[\B{t}]})_1$. Equivalently,
		\begin{equation}\label{eq:BS equa in char p}
			\Sigma'\ud{r,R}(\Z{F})\Lxf=\ssm{\alpha\in\textbf{Fun}_r,\nm{\alpha}=1}\Z{P}_{\alpha}\cdot \Z{Bd}^{[\alpha]}\B{f}^{[\alpha]}\Lxf
		\end{equation}
		for some $\Z{P}_{\alpha}\in\rcut{D_R}$.
		
		Let $I=\ideala{f_1,\dots,f_r}$. We say that a function $\Z{F}\in\cut{\fp}$ satisfies a \Nbs functional equation for $I$ if \eqref{eq:BS equa in char p} holds. All functions $\Z{F}$ satisfying a \Nbs functional equation for $I$ consist  $\ann{\cut{\fp}}N_{\mc{F}}$.
	\end{corollary}

	\section{Bernstein-Sato roots of weighted homogeneous polynomials}\label{sec:WH}
	
	In this section, we work over a perfect field $\kk$ of characteristic $p$. Fix a weight vector $\B{w}=(w_1,\dots,w_n)\in(\nz)^n$. For any $n$-tuple $\B{a}=(a_1,\dots,a_n)\in\nt^n$, we set $\wno{\B{a}}=\sum^n_{i=1}w_ia_i$. Given a polynomial
	\begin{equation*}
		f=\ssm{\B{a}\in\nt^n}c_{\B{a}}\B{x}^{\B{a}}\in\kk[\B{x}]=\kk[x_1,\dots,x_n],\quad\text{where}\;c_{\B{a}}\in\kk,
	\end{equation*}
	we define the following notation.
	\begin{itemize}
		\item $\Tupp{f}=\{\B{a}\in\nt^n:c_{\B{a}}\neq 0\}$. 
		\item $\wxd{f}=\max\{\wno{\B{a}}:\B{a}\in\Tupp{f}\}$. $\wmd{f}=\min\{\wno{\B{a}}:\B{a}\in\Tupp{f}\}$.
	\end{itemize}
	
	If $\wxd{f}=\wmd{f}=d$, $f$ is said to be \Nwh with respect to $\B{w}$ of weighted degree $d$. Usually we denote $d$ by $\wtd{f}$. We say for short that $f$ is \textbf{WH} of type $(\B{w},d)$. The sets
	\begin{equation*}
		\kk[\B{x}](\B{w},d)=\{f\in \kk[\B{x}]:\wmd{f}\geq d\}
	\end{equation*}
	indexed by $d\in\nt$ form a filtration on $\kk[\B{x}]$.

	Consider a weighted homogeneous polynomial $f$ of type $(\B{w},d)$ in $\kk[\B{x}]$ which defines an isolated singularity at the orgin of $\kk^n$. In this section, we will analyze \Nft and $F$-jumping exponents of $f$, find \Nbs functional equations for $f$, and  provide an explicit description of \Nbs roots of $f$.
	
	\subsection{$F$-thresholds and $F$-jumping exponents of polynomials}
	
	Fix an $\nt$-grading on $\kk[\B{x}]$ induced by a weighted vector $\B{w}\in(\nz)^n$. Let $\mf{m}$ be the maximal ideal $\ideala{x_1,\dots,x_n}$ of $\kk[\B{x}]$. To every $n$-tuple $\B{a}=(a_1,\dots,a_n)\in(\nz)^n$ we associate the $\mf{m}$-primary ideal $\mf{m}(\B{a})=\ideala{x_1^{a_1},\dots,x_n^{a_n}}$.
	
	\begin{theorem}\label{thm:F-the of WHideal}
		Consider the ideals
		\begin{equation*}
			I=\ideala{f_1+g_1,\dots,f_m+g_m,h_1,\dots,h_l}\quad\text{and}\quad I_0=\ideala{f_1,\dots,f_m}
		\end{equation*}
		of $\kk[\B{x}]$ contained in $\mf{m}$. Let $d_0$ (resp. $d'_0$) be the minimum of the degrees
		$\wmd{f_i}$ for all $i$ (resp. $\wmd{g_i}$, $\wmd{h_j}$ for all $i,j$.)
		\begin{enumerate}
			\item If $d'_0\geq d_0$, then the $F$-threshold $\cpt{\mf{m}(\B{a})}(I)\leq\frac{\wno{\B{a}}}{d_0}$.
			\item Suppose that every $f_i$ is \textbf{WH} of type $(\B{w},d_i)$.
			\begin{itemize}
				\item[(a)] If $\wmd{g_i}>d_i$ for all $i$, then $\cpt{\mf{m}(\B{a})}(I)\geq \cpt{\mf{m}(\B{a})}(I_0)$.
				\item[(b)] If $I_0$ is $\mf{m}$-primary, and all $d_i$ are equal to $d_0<d'_0$, then $\cpt{\mf{m}(\B{a})}(I)=\cpt{\mf{m}(\B{a})}(I_0)=\frac{\wno{\B{a}}}{d_0}.$
			\end{itemize}
		\end{enumerate}
	\end{theorem}
	
	\begin{proof}
		For every $e\in\nz$, let $A_e$ and $B_e$ denote $\nu^{\mf{m}(\B{a})}_{I_0}(p^e)$ and $\nu^{\mf{m}(\B{a})}_I(p^e)$, respectively. There exist $\B{b},\B{c}\in\nt^m$ and $\B{d}\in\nt^l$ such that 
		\begin{itemize}
			\item $|\B{b}|=A_e$, $|\B{c}|+|\B{d}|=B_e$.
			\item $\B{f}^{\B{b}}=\prod^m_{i=1}f_i^{b_i}\notin \mf{m}(\B{a})\ppm{e}$, $(\B{f+g})^{\B{c}}\B{h}^{\B{d}}=\prod^m_{i=1}(f_i+g_i)^{c_i}\prod^l_{i=1}h_i^{d_i}\notin\mf{m}(\B{a})\ppm{e}$.
		\end{itemize}
		
		To prove (1), note that
		$\mf{m}(\B{a})\ppm{e}=\mf{m}(p^e\B{a})$, there is a monomial $\B{x}^{\B{u}}$ occurring in $G$ with $\B{u}\leq p^e\B{a}-\B{1}_n$. It implies that $B_ed_0\leq\wmd{(\B{f+g})^{\B{c}}\B{h}^{\B{d}}}<p^e\wno{\B{a}}$. We obtain $\frac{B_e}{p^e}<\frac{\wno{\B{a}}}{\min\{d_0\}}$. By taking limits, we are done.
		
		To prove (2.a), we only need to show that $B_e\geq A_e$ for all $e$. As
		\begin{equation*}
			(\B{f+g})^{\B{b}}=\prod^m_{i=1}(f_i+g_i)^{b_i}=\B{f}^\B{b}+F_0,\quad\text{with}\;\wmd{F_0}\geq \wtd{\B{f}^{\B{b}}},
		\end{equation*}
		$\Tupp{(\B{f+g})^{\B{b}}}=\Tupp{\B{f}^\B{b}}\cup\Tupp{F_0}$. Since $\B{f}^\B{b}\notin\mf{m}(p^e\B{a})$, it follows that 
		\begin{equation*}
			\Tupp{(\B{f+g})^{\B{b}}},\;\Tupp{\B{f}^\B{b}}\cap\{\B{u}\in\nt^n:\B{u}\leq p^e\B{a}-\B{1}_n\}\neq\varnothing. 
		\end{equation*}
		Then we have $(\B{f+g})^{\B{b}}\notin\mf{m}(p^e\B{a})$, and $B_e\geq A_e$.
		
		To prove (2.b), note that $I_0^{A_e+1}\subseteq\mf{m}(p^e\B{a})$. Hence $\B{f}^{\B{b}}\in(\mf{m}(p^e\B{a}):I)\backslash\mf{m}(p^e\B{a})$. If $\mf{m}^{N_0}\subseteq I$, then 
		\begin{equation*}
			\kk[\B{x}](\B{w},N_1)\subseteq\mf{m}^{N_0}\subseteq I,\quad\text{where}\;N_1=N_0\max\{w_i:1\leq i\leq n\}.
		\end{equation*}
		Therefore $\B{f}^{\B{b}}\in(\mf{m}(p^e\B{a}):\kk[\B{x}](\B{w}, N_1))\backslash\mf{m}(p^e\B{a})$. There is a $\B{u}\in \Tupp{F}$ such that $\B{u}\leq p^e\B{a}-\B{1}_n$. Actually, we claim that
		$A_ed=\wtd{\B{x}^{\B{u}}}>p^e\wno{\B{a}}-|\B{w}|-N_1$. Otherwise, let $\B{v}=p^e\B{a}-\B{1}_n-\B{u}$, then $\B{x}^{\B{v}}\in\kk[\B{x}](\B{w},N_1)$. However, $p^e\B{a}-\B{1}_n\in\Tupp{\B{f}^{\B{b}}\B{x}^{\B{v}}}$, so $\B{f}^{\B{b}}\B{x}^{\B{v}}\notin\mf{m}(p^e\B{a})$, a contradiction. By taking limits on $\frac{A_e}{p^e}$ and combining (1) and (2.a), we get the conclusion.
	\end{proof}
	
	Fix an $n$-tuple $\B{a}$ of $\nz$, now we give the description of $\nu^{\mf{m}(\B{a})}_f(p^e)$ and $\cpt{\mf{m}(\B{a})}(f)$ for a weighted homogeneous polynomial $f\in\mf{m}$ which defines an isolated singularity at the origin of $\kk^n$, i.e., 
	\begin{equation}\label{eq:iso sin}
		\jac{f}=\ideala{\frac{\partial f}{\partial x_1},\dots,\frac{\partial f}{\partial x_n}}\;\text{is}\;\mf{m}\text{-primary}.
	\end{equation}
	We start with the following basic proposition.
	
	\begin{Setup}
		Every real number $\lambda\in (0,1]$ has a unique non-terminating base $p$ expansion:
		\begin{equation*}
			\lambda=\sum^{\infty}_{i=1}\frac{\lambda\dit{i}}{p^i},\quad\text{where}\;\lambda\dit{i}\in \{0,1,\dots,p-1\},
		\end{equation*}
		in which the sequence $ \{\lambda\dit{i}\}_{i\in\nz} $ is not eventually zero. Fix an integer $e\in\nz$, let $ [\lambda]_e=\sum^e_{i=1}\frac{\lambda\dit{i}}{p^i}$ denote the truncation of the non-terminating base $p$ expansion of $\lambda$ at $e$-th spot.
	\end{Setup}

	\begin{theorem}\label{thm:F-thre of f}
		Suppose that $f$ is \textbf{WH} of type $(\B{w},d)$ and $d\geq \wno{\B{a}}$. If $f$ satisfies \eqref{eq:iso sin}, then 
		\begin{equation*}
			\cpt{\mf{m}(\B{a})}(f)=\frac{\wno{\B{a}}}{d}\quad\text{or}\quad\cpt{\mf{m}(\B{a})}(f)=[\frac{\wno{\B{a}}}{d}]_L-\frac{E}{p^L}
		\end{equation*}
		for some $L\in\nz$ and $E\in\{0,\dots,n-1\}$.
	\end{theorem}
	
	\begin{proof}
		For simplicity we write $\lambda=\frac{\wno{\B{a}}}{d}$.
		We have seen above that $\cpt{\mf{m}(\B{a})}(f)\leq\lambda\leq 1$. 
		
		If $\cpt{\mf{m}(\B{a})}(f)<\lambda$, we claim that there exists an integer $L\in\nt^+$ such that $\cpt{\mf{m}(\B{a})}(f)\dit{e}=p-1$ for all $e>L$. The idea of the proof is to show that if $\cpt{\mf{m}(\B{a})}(f)\dit{e}\neq p-1$, then
		\begin{equation}\label{eq:diff lam and cpt}
			p^e[\lambda]_e-p^e[\cpt{\mf{m}(\B{a})}(f)]_e\leq n.
		\end{equation}
		From \cite[Lemma 2.7]{HZ16}, the sequence $\{p^e\tre{\lambda}-p^e\tre{\cpt{\mf{m}(\B{a})}(f)}\}_{e\in\nz}$ is non-negative, non-decreasing and unbounded. There are only finitely many $e\in\nz$ such that $\cpt{\mf{m}(\B{a})}(f)\dit{e}\neq p-1$. We choose $L$ to be the maximum of these $e$. Hence
		\begin{equation*}
			\cpt{\mf{m}(\B{a})}(f)=[\cpt{\mf{m}(\B{a})}(f)]_L+\frac{1}{p^L}=[\lambda]_L-\frac{E}{p^L},\quad\text{where}\; E=p^L[\lambda]_L-p^L[\cpt{\mf{m}(\B{a})}(f)]_L-1\in\{0,\dots,n-1\}.
		\end{equation*}
		
		Now we prove that if $\cpt{\mf{m}(\B{a})}(f)\dit{e}\neq p-1$, then \eqref{eq:diff lam and cpt} holds. Let $A_e$ denote $\nu^{\mf{m}(\B{a})}_f(p^e)$, by (2) of Lemma \ref{lem:F-thre in reg},
		$A_e=p^e\tre{\cpt{\mf{m}(\B{a})}(f)}\equiv \cpt{\mf{m}(\B{a})}(f)\dit{e}\bmod p$.
		Hence, if $\cpt{\mf{m}(\B{a})}(f)\dit{e}\neq p-1$, then $p\nmid A_e+1$. Choose $N_1\in\nz\backslash p\z$ such that $N_1\bmod p$ is the inverse of $A_e+1\bmod p$ in $\fp^\times$. 
		Applying the operator $\frac{\partial}{\partial x_i}$ to $f^{A_e+1}$, we obtain
		\begin{equation*}
			f^{A_e}\frac{\partial f}{\partial x_i}=N_1\frac{\partial f^{A_e+1}}{\partial x_i}\in \frac{\partial}{\partial x_i}(\mf{m}(p^e\B{a}))\subseteq\mf{m}(p^e\B{a}),\quad i=1,\dots,n.
		\end{equation*}
		This is because $\frac{\partial}{\partial x_i}\in \dle{\kk[\B{x}]}$. Thus $f^{A_e}\in (\mf{m}(p^e\B{a}):\jac{f})\backslash\mf{m}(p^e\B{a})$. By \cite[Setup 4.5]{HZ16},
		\begin{equation*}
			\kk[\B{x}]\left(\B{w},nd-2|\B{w}|+1\right)\subseteq \jac{f}.
		\end{equation*}
		Proceeding in a manner similar to the proof of (2.b) of Theorem \ref{thm:F-the of WHideal}, we show that
		\begin{equation}\label{eq:nu^gam_f>}
			\wtd{f^{A_e}}=A_ed\geq p^e\wno{\B{a}}+|\B{w}|-nd.
		\end{equation}
		Therefore $A_e\geq p^e[\lambda]_e-n$.
	\end{proof}
	
	\begin{remark}\label{rmk:Pro3 Spe}
		Even if $f$ is not weighted homogeneous, we still have $ \kk[\B{x}](\B{w},N_0+1)\subseteq \jac f$ for some $ N_0$. Then $\cpt{\mf{m}(\B{a})}(f)\dit{e}\neq p-1$ yields $\nu^{\mf{m}(\B{a})}_f(p^e)\wxd{f}\geq p^e\wno{\B{a}}-|\B{w}|-N_0$.
	\end{remark}
	
	This theorem is motivated by \cite[Theorem 3.5]{HZ16}. We generalize it. In \cite{HZ16}, Hern\'andez and co-authors treated the particular case when $\B{a}=\B{1}_n$ (We call $\mr{fpt}(f)=\cpt{\mf{m}(\B{1}_n)}(f)$ the $F$-pure threshold) and provided a more explicit description of values of $L$ and $E$. But they used the property that $\mr{fpt}(f)\dit{1}=\min\{\mr{fpt}(f)\dit{e}:e\in\nz\}$, which does not hold in our case. Fortunately, we do not need exact descriptions of $L$ and $E$ to compute \Nbs roots of $f$.

	\begin{theorem}\label{thm:BSR of spec f}
		Fix an $\nt$-grading on $\kk[\B{x}]$ induced by $\B{w}\in(\nz)^n$. We consider a polynomial $f\in\mf{m}$ satisfying $\eqref{eq:iso sin}$. If $\ideala{f^m}\ppdm{e}$ is a monomial ideal for all $e\in\nt$ and $m\in \textbf{N}_1(e)$, then
		\begin{equation*}
			\textbf{BSR}(f)=\{-1\}\cup\{-\cpt{\mf{m}(\B{a})}(f):\B{a}\in(\nz)^n,\wno{\B{a}}\leq \wxd{f},\cpt{\mf{m}(\B{a})}(f)\in\z_{(p)}\}.
		\end{equation*}
	\end{theorem}
	
	\begin{proof}
		As $\ideala{f^{p^e-1}}\ppdm{e}\supsetneqq\ideala{f^{p^e}}\ppdm{e}=\ideala{f}$, so $p^e-1\in\nu^\bullet_f(p^e)=\mc{B}^\bullet_f(p^e)$, then $-1$, limit of $\{p^e-1\}_{e\in\nt}$ in $\zp$, is a \Nbs root of $f$. By Theorem \ref{thm:FJ and BSR}, it is enough to show that every $F$-jumping exponent $\lambda$ of $I$ lying in the interval $(0,1)$ has the form $\cpt{\mf{m}(\B{a})}(f)$ for some $\B{a}\in(\nz)^n$ with $\wno{\B{a}}\leq\wxd{f}$. 
		
		The condition that $\ideala{f^m}\ppdm{e}$ is a monomial ideal for every $m\in \textbf{N}_1(e)$ implies that for any $c\in(0,1)$, test ideal $\tau(f,c)$ is a monomial ideal, as for sufficiently large $e$, there is $\ceil{cp^e}\leq p^e-1$ and $\tau(f,c)=\ideala{f^{\ceil{cp^e}}}\ppdm{e}$. Assume that $\textbf{FJ}(f)\cap[0,1)=\{\lambda\ud{0},\lambda\ud{1},\dots,\lambda\ud{m}\}$ with $\lambda\ud{0}=0,\lambda\ud{i}<\lambda\ud{i+1}$. For each $\lambda\ud{i}>0$, 
		there is an monomial $\B{x}^{\B{b}}\in\tau(f,\lambda\ud{i-1})\backslash\tau(f,\lambda\ud{i})$ ($\B{b}$ can be $\B{0}_n$). By \cite[Lemma 3.3]{QG21N}, we have
		\begin{equation*}
			\tau(f,\lambda\ud{i})\subseteq\mf{m}(\B{b}+\B{1}_n)\subsetneqq\tau(f,\lambda\ud{i-1}).
		\end{equation*}
		Write $\B{a}=\B{b}+\B{1}_n$. Using \cite[Corollary 2.30]{BMS08}, $\cpt{\mf{m}(\B{a})}(f)\in[\cpt{\tau(f,\lambda\ud{i-1})}(f),\cpt{\tau(f,\lambda\ud{i})}(f)]=[\lambda\ud{i-1},\lambda\ud{i}]$ is an $F$-\\jumping exponent of $I$. $\cpt{\mf{m}(\B{a})}(f)\neq \lambda\ud{i-1}$, otherwise $\tau(f,\lambda\ud{i-1})=\tau(f,\cpt{\mf{m}(\B{a})}(f))\subseteq\mf{m}(\B{a})$. So $\lambda\ud{i}=\cpt{\mf{m}(\B{a})}(f)$.
		
		Now we show that if $\lambda_i\in\z_{(p)}$, then $\wno{\B{a}}\leq \wxd{f}$. Otherwise, let
		\begin{equation*}
			C_0=\min\{\frac{\wno{\B{u}}}{\wxd{f}}-1:\B{u}\in(\nz)^n,\wno{\B{u}}>\wxd{f}\}>0.
		\end{equation*}
		By Remark \ref{rmk:Pro3 Spe}, there exists an integer $N_0$ such that if $\lambda\ud{i}\dit{e}\neq p-1$, then
		\begin{equation*}
			\nu^{\mf{m}(\B{a})}_f(p^e)-p^e\geq p^eC_0-\frac{|\B{w}|+N_0}{\wxd{f}}.
		\end{equation*}
		As $\lambda\ud{i}=\cpt{\mf{m}(\B{a})}(f)<1$, by (2) of Lemma \ref{lem:F-thre in reg}, $\nu^{\mf{m}(\B{a})}_f(p^e)\leq p^e-1$. Choose $e_0=\ceil{\log_p (\B{w}+N_0)-\log_p C_0\wxd{f}}$. Then we have $\lambda\ud{i}=[\lambda\ud{i}]_{e_0}+p^{-e_0}<1$. It yields that $[\lambda\ud{i}]_{e_0}<p^{-e_0}(p^{e_0}-1)$ and $\lambda\ud{i}\notin\z_{(p)}$.
	\end{proof}
	
	\begin{corollary}\label{cor:thmB}
		Suppose that the polynomial $f\in\kk[\B{x]}$ is \textbf{WH} of type $(\B{w},d)$. If $f$ satisfies \eqref{eq:iso sin} and $\ideala{f^m}\ppdm{e}$ is a monomial ideal for all $e\in\nt$ and $m\in \textbf{N}_1(e)$, then we have
		\begin{equation*}
			\textbf{BSR}(f)=\{-1\}\cup\{-\frac{\wno{\B{a}}}{d}:\B{a}\in(\nz)^n,\wno{\B{a}}\leq d,\cpt{\mf{m}(\B{a})}(f)=\frac{\wno{\B{a}}}{d}\in\z_{(p)}\}.
		\end{equation*}
	\end{corollary}
	
	\begin{proof}
		Theorem \ref{thm:F-thre of f} tells us that for every $\B{a}\in(\nz)^n$ such that $\wno{\B{a}}\leq d$,
		\begin{equation*}
			\cpt{\mf{m}(\B{a})}(f)=\frac{\wno{\B{a}}}{d}\quad\text{or}\quad \cpt{\mf{m}(\B{a})}(f)\notin\z_{(p)}.
		\end{equation*}
		Combining with Theorem \ref{thm:BSR of spec f}, we complete the proof.
	\end{proof}
	
	The polynomial below satisfies the additional condition of $\ideala{f^m}\ppdm{e}$ in Theorem \ref{thm:BSR of spec f}:
	\begin{equation*}
		f=\sum^n_{i=1}c_i\B{x}^{\B{a}_i}\;\text{with}\;\B{a}_i\in\nt^n,p\nmid \det(G(f)),\quad\text{where}\;G(f)=
		\begin{pmatrix}
			\B{a}_1&\cdots&\B{a}_n
		\end{pmatrix}\in M_n(\z).
	\end{equation*}
	Since $p\nmid\det(G(f))$, for every $e\in\nt$, there is an integer $n_e\in\nt$ such that $p^e\mid (n_e\det(G(f))-1)$. Let $G(f)^*\\\in M_n(\z)$ be the adjugate matrix of $G(f)$. When $m=0$, $\ideala{f^m}\ppdm{e}=\kk[\B{x}]$. When $1\leq m\leq p^e-1$, as
	\begin{equation*}
		f^m=\ssm{\B{b}\in\nt^n,|\B{b}|=m}(\prod^n_{i=1}c_i^{b_i})\B{x}^{G(f)\cdot\B{b}},
	\end{equation*}
	for any $\B{b}_1,\B{b}_2\in\nt^n$ with $|\B{b}_1|=|\B{b}_2|=m$, if $G(f)\B{b}_1-G(f)\B{b}_2\in p^e\z^n$, then
	$$n_eG(f)^*G(f)(\B{b}_1-\B{b}_2)\equiv\B{b}_1-\B{b}_2\bmod p^e\z_n.$$
	Naturally, this implies $\B{b}_1-\B{b}_2\in p^e\z_n$. Note that $|\B{b}_1|=|\B{b}_2|<p^e$, hence $\B{b}_1=\B{b}_2$. Based on this fact and Example \ref{exa:BMS09}, we conclude that
	$\ideala{f^m}\ppdm{e}$ is a monomial ideal.
	
	This condition remains strict. For general polynomials $f$, we need to find \Nbs functional equations for $f$ to estimate \Nbs roots.
	
	\subsection{Bernstein-Sato functional equations for weighted homogeneous polynomials}
	Fix an $\nt$-grading on $\kk[\B{x}]$ induced by the weight vector $\B{w}\in(\nz)^n$. For any integer $e\in\nz$, we define the \Ndpo of level $e$ on $\kk[\B{x}]$.
	\begin{align*}
		P_{\B{w},e}&=\ssm{\B{a}\in\textbf{N}_n(e)}\Z{B}_{p^{e-1}}(\wno{\B{a}}+|\B{w}|)\B{x}^{\B{a}}\dpox{(p^e-1)\B{1}_n}\B{x}^{(p^e-1)\B{1}_n-\B{a}}\\
		&=\ssm{\B{a}\in\textbf{N}_n(e),\B{a}\neq(p^e-1)\B{1}_n}\Z{B}_{p^{e-1}}(\wno{\B{a}}+|\B{w}|)\B{x}^{\B{a}}\dpox{(p^e-1)\B{1}_n}\B{x}^{(p^e-1)\B{1}_n-\B{a}} \in \dle{\kk[\B{x}]}\cdot\mf{m}.
	\end{align*}
	Here $\Z{B}_{p^{e-1}}\in\ceut{\fp}$ is the binomial polynomial (see Lemma \ref{lem:basis of cfp}). The second equality comes from the fact that $\Z{B}_{p^{e-1}}((p^e-1)|\B{w}|+|\B{w}|)=0$. $P_{\B{w},e}$ sends each monomial $\B{x}^{\B{u}}$ to $\Z{B}_{p^{e-1}}(\wno{\B{u}}+|\B{w}|)\B{x}^{\B{u}}$. Hence for any weighted homogeneous polynomial $g$, $P_{\B{w},e}(g)=\Z{B}_{p^{e-1}}(\wtd{g}+|\B{w}|)g$.

	Let $f$ be a non-constant weighted homogeneous polynomial of type $(\B{w},d)$ satisfying \eqref{eq:iso sin}. 
	$\{\B{x}^{\B{u}_i}+\jac{f}\}_{i=0}^m$ is a monomial $\kk$-basis $\{\B{x}^{\B{u}_i}+\jac{f}\}_{i=0}^m$ of $\kk[x]/\jac{f}$ as a finite-dimensional $\kk$-vector space. We assume that $\B{x}^{\B{u}_0}=1$. Let $\mc{S}$ denote the set $\{\wno{\B{u}_i}\}_{i=0}^m$ where each integer appears without repetition. As $\jac{f}$ is a homogeneous ideal in the graded ring $\kk[\B{x}]$, for each monomial $\B{x}^{\B{u}}$, either $\B{x}^{\B{u}}\in\jac{f}$ or $\wno{\B{u}}\in\mc{S}$.
	
	\begin{theorem}\label{thm:BSfun of WH}
		If $\lambda$ is a \Nbs root of $f$, then either $\lambda=-\frac{a}{b}$ with $a,b\in\nz,p\nmid b$ and $p\mid b-a$, or $\lambda=-\frac{c}{d}$ with $c\in\nz$ and $c\equiv\rho\ud{0}+|\B{w}|\bmod p$ for some $\rho\ud{0}\in\mc{S}$.
	\end{theorem}
	
	\begin{proof} 
		Resulting from Section \ref{subsec:Direct} and \cite[Proposition 7.12]{JNQ23}, we need to find those functions $\Z{F}\in\cut{\fp}$ such that $\Z{F}\bxf\in\cut{D_R}\cdot f\bxf$.
		
		\noindent\textbf{Step 1:} Fix integers $a,b\in\z$ and $k\in\nt$, $\Z{B}_{k,a,b}$ denotes the continuous function from $\zp$ to $\fp$ such that
		\begin{equation*}
			\Z{B}_{k,a,b}(\mc{x})=\Z{B}_k(a\mc{x}+b).
		\end{equation*}
		For any weighted homogeneous polynomial $g$, if $P_{\B{w},e}\cdot g\bxf=\Z{U}\bxf$ in $\cut{R_f}\bxf$ with $\Z{U}\in\cut{R_f}$, note that for any $l\in\nt$, $P_{\B{w},e}(gf^l)=\Z{B}_{p^{e-1}}(ld+\wtd{g}+|\B{w}|)gf^l$, then we have 
		\begin{equation}\label{eq:Th ufs}
			\Z{U}|_{\nt}=(\Z{B}_{p^{e-1},d,\wtd{g}+|\B{w}|}\cdot g)|_{\nt}\quad\text{and}\quad \Z{U}=\Z{B}_{p^{e-1},d,\wtd{g}+|\B{w}|}\cdot g.
		\end{equation}
		
		\noindent\textbf{Step 2:} We claim that for any monomial $\B{x}^{\B{u}}\in\kk[\B{x}]$ such that $\wno{\B{u}}\in\mc{S}$, there is
		\begin{equation}\label{eq:Step 2}
			\prod\ud{\rho\in\mc{S},\rho\geq\wno{\B{u}}}\Z{B}_{p^{e-1},d,\rho+|\B{w}|}\cdot \B{x}^{\B{u}}\bxf\in\cut{D_R}\cdot\jac{f}\bxf,\quad\text{for all}\;e\in\nz.
		\end{equation}
		We prove this by descending induction on $\wno{\B{u}}$. If $\wno{\B{u}}=\max\{\rho:\rho\in\mc{S}\}$, then $\mf{m}\cdot\B{x}^{\B{u}}\in\jac{f}$. Then by \eqref{eq:Th ufs}, there is
		\begin{equation*}
			\Z{B}_{p^{e-1},d,\wno{\B{u}}+|\B{w}|}\cdot \B{x}^{\B{u}}\bxf=P_{\B{w},e}\cdot \B{x}^{\B{u}}\bxf\in\cut{D_R}\cdot\jac{f}\bxf.
		\end{equation*}
		Now we consider the general case. For each $\B{a}\in\textbf{N}_n(e)$ such that $\B{a}\neq(p^e-1)\B{1}_n$, if $\wno{\B{u}}+\wno{(p^e-1)\B{1}_n-\B{a}}\in\mc{S}$, then by induction,
		\begin{equation*}
			\prod\ud{\rho\in\mc{S},\rho\geq\wno{\B{u}}+\wno{(p^e-1)\B{1}_n-\B{a}}}\Z{B}_{p^{e-1},d,\rho+|\B{w}|} \cdot \B{x}^{(p^e-1)\B{1}_n-\B{a}}\B{x}^{\B{u}}\bxf\in \cut{D_R}\cdot \jac{f}\bxf.
		\end{equation*}
		Multiplying with some $\Z{B}_{p^{e-1},d,\rho+|\B{w}|}$ such that $\rho\in(\wno{\B{u}},\wno{\B{u}}+\wno{(p^e-1)\B{1}_n-\B{a}})\cap\mc{S}$, we obtain
		\begin{equation*}
			\prod\ud{\rho\in\mc{S},\rho>\wno{\B{u}}}\Z{B}_{p^{e-1},d,\rho+|\B{w}|}\cdot \B{x}^{(p^e-1)\B{1}_n-\B{a}}\B{x}^{\B{u}}\bxf\in \cut{D_R}\cdot \jac{f}\bxf.
		\end{equation*}
		Otherwise, $\B{x}^{(p^e-1)\B{1}_n-\B{a}}\B{x}^{\B{u}}\in\jac{f}$. The equation above still holds. Applying $\Z{B}_{p^{e-1}}(\wno{\B{a}}+|\B{w}|)\B{x}^{\B{a}}\dpox{(p^e-1)\B{1}_n}$ on the left hand side and taking the sum for all $\B{a}\in\textbf{N}_n(e)\backslash\{(p^e-1)\B{1}_n\}$, we obtain \eqref{eq:Step 2}.
		
		\noindent \textbf{Step 3:} Based on the conclusions above, let $\B{x}^{\B{u}}=1$, then for each $e\in\nz$, there is
		\begin{equation*}
			(\Z{B}_1+1)\prod\ud{\rho\in\mc{S}}\Z{B}_{p^{e-1},d,\rho+|\B{w}|}\bxf\in \cut{D_R}(\Z{B}_1+1)\cdot\jac f\bxf\subseteq\cut{D_R}\cdot f\bxf.
		\end{equation*}
		The last containment comes from the fact that $\partial^{[1]}_{x_i}\cdot f\bxf=(\Z{B}_1+1)\cdot\frac{\partial f}{\partial x_i}\bxf$. So for any $e\in\nz$, there is
		\begin{equation*}
			(\Z{B}_1+1)\prod\ud{\rho\in\mc{S}}\Z{B}_{p^{e-1},d,\rho+|\B{w}|}\in\ann{\cut{\fp}}\frac{\cut{D_R}\cdot\bxf}{\cut{D_R}\cdot f\bxf}.
		\end{equation*}
		
		Let $\lambda$ be a \Nbs root of $f$. Then by Theorem \ref{thm:FJ and BSR}, $\lambda$ has the form $-\frac{a}{b}$ with $a,b\in\nz,(a,b)=1$, $b>a$ and $p\nmid b$. By (1) of Corollary \ref{cor:Basic pro of BS-roots}, $(\Z{B}_1+1)\prod\ud{\rho\in\mc{S}}\Z{B}_{p^{e-1},d,\rho+|\B{w}|}$ vanishes at $\lambda$ for all $e\in\nz$. Note that $\Z{B}_1(\lambda)=(-a\bmod p)(b\bmod p)^{-1}$. If $p\nmid b-a$, then $\Z{B}_1(\lambda)+1\neq 0$. Then for any $e\in\nt$, there exists an $\rho\ud{e}\in\mc{S}$ such that $\Z{B}_{p^e}(\lambda d+\rho\ud{e}+|\B{w}|)=0$.
		For each $\rho\in\mc{S}$, write $\lambda\ud{\rho}=\lambda d+\rho+\B{w}$. Assume that the $p$-adic expression of $\lambda d$ (resp. $\lambda\ud{\rho}$) in $\zp$ is 
		\begin{equation*}
			\lambda d=\ssm{i\in\nt}(\lambda d)^{[i]}p^i,\quad \text{resp.}\; \lambda\ud{\rho}=\ssm{i\in\nt}\lambda\ud{\rho}^{[i]}p^i,\quad\text{with}\;\lambda^{[i]},\lambda\ud{\rho}^{[i]}\in\{0,\dots,p-1\}.
		\end{equation*}
		We choose an integer $e_0\in\nt$ such that all $\{\rho+|\B{w}|:\rho\in\mc{S}\}\subseteq \textbf{N}_1(e_0)$.
		
		\noindent\textbf{Case 1:} There is an integer $e_1>e_0$ such that $\lambda^{[e_1]}\leq p-2$, hence $\lambda\ud{\rho}^{[i]}=(\lambda d)^{[i]}$ for all $i>e_1$ and $\rho\in\mc{S}$. $\Z{B}_{p^e}(\lambda\ud{\rho\ud{e}})=0$ for all $e$ implies that $(\lambda d)^{[i]}=0$ for all $i>e_1$. So $\lambda d\in\nt$, but this is impossible, since $\lambda<0$.
		
		\noindent\textbf{Case 2:} For all $i>e_0$, $(\lambda d)^{[i]}=p-1$. In this situation, $\lambda d=\ideala{\lambda d}_{e_0}-p^{e_0+1}\in-\nt$. We write $\lambda=-\frac{c}{d}$ with $0<c<d$. Then $\Z{B}_1(\lambda\ud{\rho\ud{0}})=0$ yields that $p\mid c-\rho_0-|\B{w}|$.
	\end{proof}

	\section{Thom-Sebastiani properties of Bernstein-Sato roots}\label{sec:TS pro}
	In this section, we take $F$-finite regular rings $R$ and $S$ of characteristic $p$. Let $I$ be an ideal of $R$ and $J$ an ideal of $S$. $\mf{B}=I\ott{\fp}J$ can be seen as an ideal of $Q=R\ott{\fp}S$, generated by $\{f\otimes g\}_{f\in I,g\in J}$. We define $\mf{A}\\=\ideala{I,J}$ to be the ideal $I\ott{\fp}S+R\ott{\fp}J$ of $Q$.

	\begin{lemma}\label{lem:diop on tensor}
		Under the assumptions above, for each $e\in\nt$, there is a $Q$-module isomorphism 
		\begin{equation*}
			\textbf{T}^e_{R,S}:\dle{R}\ott{\fp}\dle{S}\mapsto \dle{Q},\quad \textbf{T}^e_{R,S}(\phi\otimes\psi):f\otimes g\mapsto \phi(f)\otimes\phi(g),
		\end{equation*}
		where the $Q$-module structure on $\dle{R}\ott{\fp}\dle{S}$ is given by 
		\begin{equation*}
			(r\otimes s)(\phi\otimes\psi)=(r\phi)\otimes(s\psi),\quad \text{for all}\;r\in R,s\in S,\phi\in\dle{R},\psi\in\dle{S}.
		\end{equation*}
		These $Q$-module isomorphisms induce a $Q$-algebra isomorphism $\textbf{T}_{R,S}:D_R\ott{\fp}D_S\to D_Q$.
	\end{lemma}
	
	\begin{proof}
		By Kunz's theorem, for each $e\in\nt$, $R$ is $R^{p^e}$-flat. $R^{p^e}$ is Noetherian, so $R$ is a locally free $R^{p^e}$-module. Let $u_1,\dots,u_n\in R^{p^e}$ be such that $\{u_i\}_{i=1}^n$ generate the unit ideal of $R^{p^e}$, and $R_{u_i}$ is $(R^{p^e})_{u_i}$-free. Write $u_i$ as $\tld{u}_i^{p^e}$ for some $u_i\in R$. $\{\tld{u}_i\}_{i=1}^n$ generate the unit ideal of $R$, and $R_{\tld{u}_i}=R_{u_i}$ is free over $(R^{p^e})_{u_i}=(R_{\tld{u}_i})^{p^e}$.
		Repeat the argument for $S$, let $v_1,\dots,v_m\in S^{p^e}$ be such that $\{v_i\}_{i=1}^m$ generate the unit ideal of $S^{p^e}$, and $S_{v_i}$ is $(S^{p^e})_{v_i}$-free. Write $v_i=\tld{v}_i^{p^e}$ for some $\tld{v}_i\in S$. $\{\tld{v}_i\}_{i=1}^m$ generate the unit ideal of $S$, and $S_{\tld{v}_j}$ is $(S_{\tld{v}_j})^{p^e}$-free. Note that $\{w_{ij}=\tld{u}_i\otimes \tld{v}_j\}_{1\leq i\leq n,1\leq j\leq m}$ generate the unit ideal of $Q$, it suffices to show that for every $i$ and $j$,
		\begin{equation*}
			(\textbf{T}^e_{R,S})_{w_{ij}}:Q_{w_{ij}}\ott{Q}(\dle{R}\ott{\fp}\dle{S})\to Q_{w_{ij}}\ott{Q}\dle{Q}
		\end{equation*}
		is an isomorphism. From the ring isomorphism
		\begin{equation}\label{eq:Rf ot Sg}
			R_{\tld{u}_i}\ott{\fp}S_{\tld{v}_j}\cong Q_{w_{ij}},\quad \frac{f}{\tld{u}_i^{n_1}}\otimes\frac{g}{\tld{v}_j^{m_1}}\mapsto\frac{\tld{u}_i^{m_1}g\otimes \tld{v}_j^{n_1}g}{(\tld{u}_i\otimes \tld{v}_j)^{n_1+m_1}},
		\end{equation}
		we obtain the following commutative diagram
		\begin{equation*}
			\begin{tikzcd}
				Q_{w_{ij}}\ott{Q}(\dle{R}\ott{\fp}\dle{S})& &Q_{w_{ij}}\ott{Q}\dle{Q}\\
				\dle{R_{\tld{u}_i}}\ott{\fp}\dle{S_{\tld{v}_j}}& &\dle{R_{\tld{u}_i}\ott{\fp}S_{\tld{v}_j}}.
				\arrow["(\textbf{T}^e_{R,S})_{w_{ij}}",from=1-1,to=1-3]
				\arrow["\textbf{T}^e_{R_{\tld{u}_i},S_{\tld{v}_j}}",from=2-1,to=2-3]
				\arrow["\cong",from=1-1,to=2-1]
				\arrow["\cong",from=1-3,to=2-3]
			\end{tikzcd}
		\end{equation*}
		The left vertical arrow follows from Lemma \ref{lem:local of diff} and \eqref{eq:Rf ot Sg}:
		\begin{align*}
			Q_{w_{ij}}\ott{Q}(\dle{R}\ott{\fp}\dle{S})\cong&
			(R_{\tld{u}_i}\ott{\fp}S_{\tld{v}_j})\ott{Q}(\dle{R}\ott{\fp}\dle{S})\\ 
			\cong&(R_{\tld{u}_i}\ott{R}\dle{R})\ott{\fp}(S_{\tld{v}_j}\ott{S}\dle{S})\cong\dle{R_{\tld{u}_i}}\ott{\fp}\dle{S_{\tld{v}_j}}.
		\end{align*}
		Similarly for the right vertical arrow. The lower horizontal arrow is an isomorphism (\cite[Lemma 2.56]{QuinPHD}). Hence, we complete the proof of the lemma.
	\end{proof}

	\begin{theorem}\label{mthm:C}
		We keep the notation introduced at the beginning of this section.
		\begin{enumerate}
			\item$\textbf{BSR}(\mf{A})=\{\mc{x}+\mc{y}:\mc{x}\in\textbf{BSR}(I),\mc{y}\in\textbf{BSR}(J)\}$.
			\item Suppose that $I\not\subseteq\mr{nil}(R)$ and $J\not\subseteq\mr{nil}(S)$. Then $\textbf{BSR}(\mf{B})=\textbf{BSR}(I)\cup\textbf{BSR}(J)$.
		\end{enumerate}
	\end{theorem}
	
	\begin{proof}
		By \cite[Lemma 3.11]{JNQ23}, for each $e\in\nt$, there is
		\begin{equation*}
			\mc{B}^\bullet_I(p^e)=\{n\in\dle{R}\cdot I^n\supsetneqq\dle{R}\cdot I^{n+1}\}\supseteq\mc{B}^\bullet_I(p^{e+1})=\{n\in\textbf{D}^{e+1}_R\cdot I^n\supsetneqq\textbf{D}^{e+1}_R\cdot I^{n+1}\}.
		\end{equation*}
		So $\mc{x}\in\zp$ is a \Nbs root of $I$ if and only if there is an infinite subset $\{e_i\}_{i\in\nt}$ of $\nt$ and $\nu\ud{e_i}\in\mc{B}^\bullet_I(p^{e_i})$ such that $\mc{x}$ is the limit of $\{\nu\ud{e_i}\}_{i\in\nt}$ in $\zp$. The same conclusion applies to $J$, $\mf{A}$ and $\mf{B}$.
		
		To prove $(1)$, we show that for each $e\in\nt$, there is $\mc{B}^\bullet_\mf{A}(p^e)=\mc{B}^\bullet_I(p^e)+\mc{B}^\bullet_J(p^e)$. It implies that $\textbf{BSR}(\mf{A})\supseteq\textbf{BSR}(I)+\textbf{BSR}(J)$. As $\mf{A}^n=\sum^n_{a=0} I^a\ott{\fp}J^{n-a}$, by Lemma \ref{lem:diop on tensor}, 
		\begin{equation}\label{eq:DQmf{a}}
			\dle{Q}\cdot\mf{A}^n=\sum^n_{a=0}(\dle{R}\ott{\fp}\dle{S})\cdot (I^a\ott{\fp}J^{n-a})=\sum^n_{a=0}(\dle{R}\cdot I^a)\ott{\fp}(\dle{S}\cdot J^{n-a}).
		\end{equation}
		If for every $0\leq a\leq n$, either $a\notin\mc{B}^\bullet_I(p^e)$ or $n-a\notin\mc{B}^\bullet_J(p^e)$, then 
		\begin{equation*}
			(\dle{R}\cdot I^a)\ott{\fp}(\dle{S}\cdot J^{n-a})\subseteq (\dle{R}\cdot I^{a+1})\ott{\fp}(\dle{S}\cdot J^{n-a})+(\dle{R}\cdot I^a)\ott{\fp}(\dle{S}\cdot J^{n+1-a})
		\end{equation*}
		By \eqref{eq:DQmf{a}} we have $n\notin\mc{B}^\bullet_{\mf{A}}(p^e)$. Hence $\mc{B}^\bullet_{\mf{A}}(p^e)\subseteq\mc{B}^\bullet_I(p^e)+\mc{B}^\bullet_J(p^e)$. In the other direction, for any integers $n_1\in\\\mc{B}^\bullet_I(p^e)$ and $n_2\in\mc{B}^\bullet_J(p^e)$, we choose $u\in\dle{R}\cdot I^{n_1}\backslash\dle{R}\cdot I^{n_1+1}$ and $v\in\dle{S}\cdot J^{n_2}\backslash \dle{S}\cdot J^{n_2+1}$. Note that 
		\begin{equation*}
			\dle{Q}\cdot \mf{A}^{n_1+n_2+1}\subseteq\ideala{\dle{R}\cdot I^{n_1+1},\dle{S}\cdot J^{n_2+1}},
		\end{equation*}
		thus $u\otimes v\in\dle{Q}\cdot \mf{A}^{n_1+n_2}\backslash\dle{Q}\cdot \mf{A}^{n_1+n_2+1}$. Otherwise, the image of $u\otimes v$ in $Q/\ideala{\dle{R}\cdot I^{n_1+1},\dle{S}\cdot J^{n_2+1}}\cong\\(R/\dle{R}\cdot I^{n_1+1})\ott{\fp}(S/\dle{S}\cdot J^{n_1+1})$ is zero. But this is impossible. Therefore, $n_1+n_2\in\mc{B}^\bullet_\mf{A}(p^e)$.

		Let $\mc{z}$ be a \Nbs root of $\mf{A}$. There is a sequence $\{\nu\ud{e}\}_{e\in\nt}$ with $\nu\ud{e}\in\mc{B}^\bullet_{\mf{A}}(p^e)$ such that $\nu\ud{e}$ converges to $\mc{z}$. Each $\nu\ud{e}$ has the form $\nu\ud{I,e}+\nu\ud{J,e}$ with $\nu\ud{I,e}\in\mc{B}^\bullet_I(p^e)$ and $\nu\ud{J,e}\in\mc{B}^\bullet_J(p^e)$. As $\zp$ is a compact metrizable space, $\{\nu\ud{I,e}\}_{e\in\nt}$ has a convergent subsequence $\{\nu\ud{I,e_i}\}_{i\in\nt}$. If $\nu\ud{I,e_i}$ converges to $\mc{x}\in\zp$, then $\mc{x}\in\textbf{BSR}(I)$, and $\nu\ud{J,e_i}$ converges to $\mc{z}-\mc{x}\in\textbf{BSR}(J)$. And we get the conclusion $\textbf{BSR}(\mf{A})\subseteq\textbf{BSR}(I)+\textbf{BSR}(J)$.
		
		To prove (2), it suffices to demonstrate that $\mc{B}^\bullet\ud{\mf{B}}(p^e)=\mc{B}^\bullet\ud{I}(p^e)\cup\mc{B}^\bullet\ud{J}(p^e)$ for all $e\in\nt$. As $\mf{B}^n=I^n\ott{\fp}J^n$. By applying Lemma \ref{lem:diop on tensor}, we have
		\begin{equation*}
			\dle{Q}\cdot \mf{B}^n=(\dle{R}\ott{\fp}\dle{S})\cdot \mf{B}^n=(\dle{R}\cdot I^n)\ott{\fp}(\dle{S}\cdot J^n).
		\end{equation*}
		It follows that $\mc{B}^\bullet\ud{\mf{B}}(p^e)\subseteq\mc{B}^\bullet\ud{I}(p^e)\cup\mc{B}^\bullet\ud{J}(p^e)$. 
		
		Now we show the inclusion in the other direction. Given $n_1\in\mc{B}^\bullet\ud{I}(p^e)$, in fact, $\dle{R}\cdot I^{n_1+1}$ and $\dle{S}\cdot J^{n_1}$ are 
		non-zero $\fp$-flat modules, we obtain the following injective homomorphisms
		\begin{equation*}
			(\dle{R}\cdot I^{n_1+1})\ott{\fp}(\dle{S}\cdot J^{n_1+1})\to(\dle{R}\cdot I^{n_1+1})\ott{\fp}(\dle{S}\cdot J^{n_1})\to(\dle{R}\cdot I^{n_1})\ott{\fp}(\dle{S}\cdot J^{n_1}).
		\end{equation*}
		The cokernel of the second homomorphism is given by
		\begin{equation*}
			\frac{\dle{R}\cdot I^{n_1}}{\dle{R}\cdot I^{n_1+1}}\ott{\fp}(\dle{S}\cdot J^{n_1})\neq 0.
		\end{equation*}
		This implies that $(\dle{R}\cdot I^{n_1})\ott{\fp}(\dle{S}\cdot J^{n_1})\supsetneqq(\dle{R}\cdot I^{n_1+1})\ott{\fp}(\dle{S}\cdot J^{n_1+1})$. So $n_1\in\mc{B}^\bullet\ud{\mf{B}}(p^e)$. By applying the same approach to $m_1\in\mc{B}^\bullet\ud{J}(p^e)$, we can show that $m_1\in\mc{B}^\bullet\ud{\mf{B}}(p^e)$. Hence $\mc{B}^\bullet\ud{\mf{B}}(p^e)\supseteq\mc{B}^\bullet\ud{I}(p^e)\cup\mc{B}^\bullet\ud{J}(p^e)$.
	\end{proof}
	
	\bibliography{A3bio.bib}
	\bibliographystyle{alpha}
	
	Department of Mathematical Sciences, Tsinghua University, Beijing, 100084, P. R. China.
	
	\textit{Email address}: \texttt{taosy22@mails.tsinghua.edu.cn}
	
	Department of Mathematical Sciences, Tsinghua University, Beijing, 100084, P. R. China.
	
	\textit{Email address}: \texttt{xiaozd21@mails.tsinghua.edu.cn}
	
	Department of Mathematical Sciences, Tsinghua University, Beijing, 100084, P. R. China.
	
	\textit{Email address}: \texttt{hqzuo@mail.tsinghua.edu.cn}
\end{document}